\documentclass[reqno,a4paper,12pt]{amsart} 
\usepackage{amsmath,amscd,amsfonts,amssymb}
\usepackage{mathrsfs,dsfont}
\usepackage{enumerate}
\usepackage{hyperref}

\numberwithin{equation}{section}
\numberwithin{figure}{section}

\def\R{\mathbb{R}}

\def\Z{\mathbb{Z}}

\def\T{\mathbb{T}}

\def\D{\mathcal{D}}

\def\1{\mathds{1}}

\def\dH{\dim_{\mathcal{H}}}

\newcommand{\RD}{\text{RapDec}}

\renewcommand\leq{\leqslant}
\renewcommand\geq{\geqslant}

\renewcommand\hat{\widehat}

\newcommand{\supp}{\operatorname{supp}}

\newcommand{\dist}{\operatorname{dist}}

\theoremstyle{plain}
\newtheorem{thm}{Theorem}[section]
\newtheorem{lem}[thm]{Lemma}

\newtheorem{prop}[thm]{Proposition}

\newtheorem*{claim*}{Claim}
\newtheorem*{thm*}{Theorem}

\theoremstyle{definition}

\newtheorem*{definition*}{Definition}
\newtheorem*{remarks*}{Remarks}
\newtheorem*{remark*}{Remark}

\newenvironment{enumerate-math}
{\begin{enumerate}
\addtolength{\itemsep}{5pt}
}
{\end{enumerate}}

\newenvironment{enumerate-text}
{\begin{enumerate}
\addtolength{\itemsep}{5pt}
}
{\end{enumerate}}

\begin{document}

\title[Pinned distance set and Mizohata-Takeuchi-type estimates]{Lebesgue measure of distance sets with regular pins and multi-scale Mizohata-Takeuchi-type estimates}

\author{Bochen Liu}
\email{Bochen.Liu1989@gmail.com}
\address{Department of Mathematics \& International Center for Mathematics, Southern University of Science and Technology, Shenzhen 518055, China}

\thanks{This work is supported by the National Key R\&D
Program of China 2024YFA1015400, and the National Natural Science Foundation of
China grant 12131011.}
\date{}


\begin{abstract}
Suppose $E, F$ are Borel sets in the plane, $\dim_{\mathcal{H}} E>1$, $\dim_{\mathcal{H}} E+\dim_{\mathcal{H}} F>2$, and $F$ has equal Hausdorff and packing dimension. We prove that there exists $y\in F$ such that the pinned distance set
$$\Delta_y(E):=\{|x-y|:x\in E\}$$
has positive Lebesgue measure. In particular, it settles the regular case of the distance set problem in the plane. The main ingredients of the proof consist of a multi-scale Good-Bad decomposition and a multi-scale Mizohata-Takeuchi-type estimate with arbitrary small power-loss.
\end{abstract}
\maketitle

\tableofcontents

\section{Introduction}
\subsection{Falconer distance conjecture}

Falconer distance conjecture \cite{Fal85} is one of the most famous open problems in harmonic analysis and geometric measure theory since 1985. It states that, if a Borel set $E\subset\R^d$, $d\geq 2$, has Hausdorff dimension $\dH E> d/2$, then its distance set
$$\Delta(E):=\{|x-y|: x,y\in E\}$$
has positive Lebesgue measure (denoted by $|\Delta(E)|>0$). A stronger version is the pinned distance set problem, first studied by Peres and Schlag \cite{PS00} in 2000, that asks whether there exists $y\in E$ such that the pinned distance set
$$\Delta_y(E):=\{|x-y|: x\in E\}$$
has positive Lebesgue measure when $\dH E>d/2$.

With more than 40 years passed, accumulating efforts of many excellent mathematicians, Falconer distance conjecture is still open in every dimension. So far, there are two perspectives to study this problem. One is to investigate how large $\dH E$ needs to be to ensure $\Delta(E)$ or $\Delta_y(E)$ has positive Lebesgue measure. Falconer's first result \cite{Fal85} in 1985 is of this type, which is a $L^\infty$-estimate in today's point of view. In 1987, Mattila \cite{Mat87} proposed his $L^2$-approach for $|\Delta(E)|$, now called the Mattila integral, that reduced the problem to spherical averaging operators. Since then, Bourgain \cite{Bou94}, Wolff \cite{Wol99}, Erdogan \cite{Erd05}, Du-Guth-Ou-Wang-Wilson-Zhang \cite{DGOWWZ21}, Du-Zhang \cite{DZ18} have used the newly developed tools in harmonic analysis at the time to improve spherical averaging estimates for partial results on Falconer distance conjecture. 

In 2019, the author \cite{Liu18} proposed an $L^2$-approach for the pinned distance set problem. Shortly after, Guth, Iosevich, Ou and Wang \cite{GIOW18} proved the current best-known result on Falconer distance conjecture in the plane:
$$\dH E>5/4\implies |\Delta_y(E)|>0, \text{ for some }y\in E,$$
building on the author's approach and an insightful use of the decoupling theory. Their argument in fact shows 
$$\dH F>1,\  3\dH E+ \dH F> 5\implies |\Delta_y(E)|>0, \text{ for some }y\in F.$$
We shall review their proof strategy in Section \ref{subsec-Good-Bad-GIOW}. This result was later extended to all even dimensions by Du-Iosevich-Ou-Wang-Zhang \cite{DIOWZ20}. Following the same idea with more careful analysis, Du-Ou-Ren-Zhang \cite{DORZ23-2}\cite{DORZ23-1} obtained the current best-known results in dimension $3$ and higher. 

The other perspective on Falconer distance conjecture is to compute the Hausdorff dimension of $\Delta(E), \Delta_y(E)$ assuming $\dH E> d/2$. The first nontrivial result on $\Delta(E)$ is due to Katz-Tao \cite{KT01} and Bourgain \cite{Bou03} from discretized sum-product estimates. Its pinned version is due to Shmerkin \cite{Shm20}. The first nontrivial explicit dimensional exponent is due to Keleti and Shmerkin \cite{KS18}, who proposed the novel idea of scale-selection. A partial result based on estimates of Guth-Iosevich-Ou-Wang was given by the author \cite{Liu18-dimension}, that was later extended to all even dimensions by Wang-Zheng \cite{WZ22}.

 The current record in the plane is due to Stull \cite{Stull22+}, who proved 
 $$\dH E>1\implies \dH \Delta_y(E)\geq \frac{3}{4}, \text{ for some }y\in E.$$
 His method is beyond the author's understanding (as well as the recent work of Fiedler-Stull \cite{FS23+}\cite{FS25+}), but according to Hong Wang (personal communication), she and Shmerkin understand the key idea of Stull \cite{Stull22+} and are able to obtain the same dimensional exponent $\frac{3}{4}$ via incidence estimates. We also refer to the recent work of Shmerkin and Wang \cite{SW25} for partial results on the endpoint case $\dH E\geq d/2$ for all $d\geq 2$.
 
 Recently, the less well-known Assouad dimension $\dim_A$ has drawn much attention. The problem on $\dim_A\Delta(E)$ in terms of $\dim_A E$ has been solved by Fraser in the plane \cite{Fraser18}\cite{Fraser23}. There is also a very recent work of Fraser-Pham \cite{FP26+} on the relation between the Fourier dimension of $E$ and the Hausdorff dimension of $\Delta(E)$.
 
 Though open in general, the distance set problem is better understood on sets with regularity. B\'ar\'any \cite{Barany17}, Orponen \cite{Orponen12} studied self-similar sets, Ferguson-Fraser-Sahlsten \cite{FFS15} studied self-affine sets, Fraser-Pollicott \cite{FP15} studied self-conformal sets. A breakthrough in bypassing dynamical systems was Orponen's work \cite{Orp17} on Ahlfors-David regular sets, and the idea was further developed by Shmerkin \cite{shm17}\cite{Shm18}, Keleti-Shmerkin \cite{KS18}. The final version can be stated as the following: suppose $E, F$ are Borel sets in the plane, then
 $$\dH F>1, \ \dH E=\dim_P E>1\implies \dH \Delta_y(E)=1, \text{ for some }y\in E,$$ where $\dim_P$ denotes the packing dimension (see Section \ref{sec-prelim} below for definition). In particular, if $E\subset\R^2$ has equal Hausdorff and packing dimension $>1$, then $\dH \Delta_y(E)=1$ for some $y\in E$.

``$\dH \Delta(E)=1$" is exciting and already answers the original question of Falconer in \cite{Fal85}, but it is still weaker than ``positive Lebesgue measure". There is very little discussion on Lebesgue measure of distance sets of regular sets in the literature. The only work we know is due to Iosevich and the author \cite{IL16} on $|\Delta(A_1\times A_2)|$ with $A_1, A_2\subset\R$, Ahlfors-David regular, $\dim A_1+\dim A_2>\frac{4}{3}-\epsilon_0$.

 In this paper, we settle the regular case of the distance set problem in the plane.

\begin{thm}
	\label{thm-main}
	Suppose $E, F$ are Borel sets in the plane, $\dH E>1$, $\dH E+\dH F>2$, and $F$ has equal Hausdorff and packing dimension. Then there exists $y\in F$ such that the pinned distance set
$$\Delta_y(E):=\{|x-y|:x\in E\}$$
has positive Lebesgue measure. In particular, if $E\subset\R^2$ has equal Hausdorff and packing dimension $>1$, then $|\Delta_y(E)|>0$ for some $y\in E$.
\end{thm}

Notice that there exists $E\subset\R^2$, $\dH E=\dim_P E=1$, while $|\Delta(E)|=0$. For example, one can take $E=A\times\{0\}$ with
$$A:=\bigcap_{i=1}^\infty \{a\in[0,1]:\dist(a,q_i^{-1}\Z )<q_i^{-1-\epsilon_i}\},$$
where $q_{i+1}>q_i^i>1$ and $\epsilon_i\downarrow 0$ such that $10q_i^{-\epsilon_i}\rightarrow 0$. Then $\dH A=\dim_P A=1$ (see, for example, Section 8.5 in \cite{Fal86} for $\dH A\geq 1-\epsilon_i$, $\forall\,i$, and it is well known that $\dH A\leq \dim_P A\leq 1$). Therefore $\dH E=\dim_P E=1$, while
$$|\Delta(E)|=|A-A|\leq 10q_i^{-\epsilon_i}\rightarrow 0.$$

The main ingredients of the proof of Theorem \ref{thm-main} consist of a multi-scale Good-Bad decomposition and a multi-scale Mizohata-Takeuchi-type estimate with arbitrary small power-loss. The single-scale Good-Bad decomposition was proposed by Guth-Iosevich-Ou-Wang \cite{GIOW18}. Our multi-scale Good-Bad decomposition is inspired by them, but not a direct generalization (see the discussion in Section \ref{subsec-Good-Bad}). As a result of the difference, it is the first time in recent works that $\dH F\leq 1$ is allowed for $|\Delta_y(E)|>0$ or $\dim\Delta_y(E)=1$, $y\in F$. Also, unlike the refined decoupling inequality in \cite{GIOW18}, there is no deep $L^p$ theory in this paper. We develop an $L^2$ ball inflation argument that has its own interest (see Section \ref{sec-L2}). It also leaves plenty of room for more tools to play a role in this framework.

\subsection{Mizohata-Takeuchi conjecture}
Let $\Sigma$ be a compact $C^2$ hypersurface in $\R^d$ with surface measure $\sigma$. The Mizohata-Takeuchi conjecture arose in the study of dispersive PDE from 1970s. It states that, for every non-negative weight $w$ on $\R^d$,
$$\int_{\R^d}|\widehat{f\,d\sigma}(x)|^2w(x)dx\lesssim \sup_T w(T) \int_{\Sigma} |f|^2d\sigma,$$
where $w(T):=\int_T w$ and the supremum is taken over all 1-neighborhoods of doubly-infinite lines. When $\Sigma$ is a hyperplane, it simply follows from Plancherel. On the other hand, unfortunately, it has been proved false on every hypersurface that is
not part of a hyperplane, by a recent delicate example of Cairo \cite{Cairo25+}. 
For more information on the history, related works, and discussion on the Mizohata-Takeuchi conjecture, we refer to the inspiring talk of Guth \cite{Guth22}, comprehensive introductions of Carbery-Iliopoulou-Wang \cite{CIW24}, Cairo \cite{Cairo25+}, Cairo-Zhang \cite{CZ25+}, as well as recent works of Carbery-H\"anninen-Valdimarsson \cite{CHV23}, Carbery-Li-Pang-Yung \cite{CLPY25+}, Mulherkar \cite{Mulherkar25+}, and references therein.

One way to describe how far it is from being true is to consider its local version: 
\begin{equation}\label{local-MT}\int_{B_R}|\widehat{f\,d\sigma}(x)|^2w(x)dx\leq CR^\alpha\cdot \sup_T w(T) \int_{\Sigma} |f|^2d\sigma.\end{equation}
Though the original Mizohata-Takeuchi conjecture fails in general, it is an open problem that whether \eqref{local-MT} holds for every $\alpha>0$, called the local Mizohata-Takeuchi conjecture. In \cite{CIW24}, Carbery, Iliopoulou and Wang proved \eqref{local-MT} for every $\alpha>\frac{d-1}{d+1}$ on strictly convex $C^2$ hypersurfaces. Whether $\frac{d-1}{d+1}$ is optimal is not known, but Guth \cite{Guth22} pointed out that a $R^{\frac{d-1}{d+1}-\epsilon}$-power-loss is inevitable under the usual ``wave packet decomposition axioms". See \cite{CIW24} for a detailed discussion on Guth's outline.

In a very recent work of Cairo-Zhang \cite{CZ25+}, a family of counterexamples are constructed, that in particular shows the exponent of Carbery-Iliopoulou-Wang is sharp for many strictly convex $C^2$ hypersurfaces up to the endpoint. But whether it is sharp for specific surfaces, even for familiar ones like the paraboloid or the sphere, is still unknown.

We call \eqref{local-MT} ``single-scale" as all tubes under consideration have the same size. In this paper, we find that some multi-scale Mizohata-Takeuchi-type estimates have no power-loss, up to an arbitrary small constant $\delta>0$. The one we need for Theorem \ref{thm-main} is quite complicated to state. Instead,  we state the following baby version in the introduction. This observation is also where this project started. 

From now we only consider the case $\Sigma=S^{d-1}$, because our motivation is on the distance set problem. The proof also works on more general surfaces for associated results.

\begin{prop}
    \label{thm-warmup}
    Let $\sigma$ denote the surface measure on $S^{d-1}$. Suppose $w$ is a  non-negative weight on $B_R\subset\R^d$, $0<\delta\ll 1/2$ is a constant, and $$R^{\delta}=r_0<\cdots<r_M=R$$
    is a sequence with $r_{j+1}\leq r_j^2$. Then
    $$w(B_R)^{-1}\int_{B_R}|\widehat{f\,d\sigma}|^2w\lesssim_{d, \delta, M} R^{C\delta}\left(\prod_{j=0}^{M-1} \sup_{\substack{\forall\,Q_{j+1}\in\D_{r_{j+1}}\\\forall\,T_j\in\T_{r_j, r_{j+1}}}} w_{Q_{j+1}}(T_j)\right) \|f\|_{L^2(S^{d-1})}^2,$$
where $\D_{r_{j+1}}$ denotes the collection of almost disjoint $r_{j+1}$-cubes, $\T_{r_j, r_{j+1}}$ denotes the collection of $r_j\times\cdots\times r_j\times r_{j+1}$ tubes, $w(B_R):=\int_{B_R} w$,  $w_Q(T):=w(Q)^{-1}\int_{T\cap Q} w$, and $C>0$ is a constant independent in $R, M, \delta$.
\end{prop}
When $w=\chi_{B_R}$, this result becomes 
$$\|\widehat{f\,d\sigma}\|_{L^2(B_R)}^2\lesssim_{d,\delta} R^{C\delta}\cdot R\|f\|_{L^2(S^{d-1})}^2,$$
which coincides with the classical sharp estimate of Agmon-H\"ormander \cite{AH76} up to a factor $R^{C\delta}$. The constrain $r_{j+1}\leq r_j^2$ in Proposition \ref{thm-warmup} is due to technical reasons. This is also why we need some regularity on the pin set $F$ in Theorem \ref{thm-main}. I wonder if it is necessary.

A heuristic proof of Proposition \ref{thm-warmup} is given in Section \ref{sec-proof-warmup} to illustrate the idea of our $L^2$ ball inflation. A detailed discussion on the version for Theorem \ref{thm-main} is given in Section \ref{subsec-L2-good}.

\subsection*{Notation}
 We list notation used in this paper to make it easier for readers to recall.

$X\lesssim Y$ means $X\leq CY$ for some constant $C>0$; $X\lesssim_\delta Y$ means $X\leq C_\delta Y$ for some constant $C_\delta>0$ depending on $\delta>0$. We omit the subindex $\delta$ in the proof because $\delta>0$ is a fixed constant. The constant $C>0$ in $R^{C\delta}$ may vary from line to line but is independent in $R$ and $\delta$. All numbers in this paper can be assumed dyadic.

$\RD(R)$ means it is $\leq C_NR^{-N}$ for arbitrary large $N>0$, called a negligible error in the proof.

$\D_r$ denotes the collection of almost disjoint half-open-half-closed $r$-cubes that cover $\R^d$. We may assume all cubes in this paper are dyadic. For a cube $Q$, denote by $2Q$ its dilation by factor $2$ from the center. Once the collection $\D_r$ is clear, denote by $Q(y)$ the $Q\in\D_r$ that contains $y$.

Let $N(E, r)$ denote the smallest number of $r$-balls needed to cover $E$.

Let $m$ denote the Lebesgue measure. We also use $|A|:=m(A)$ to denote the Lebesgue measure of a set $A$.

Let $\sigma_r$ denote the normalized surface measure on $rS^{d-1}$. Also denote $\sigma=\sigma_1$.

For a measure $\nu$, denote by $\nu_Q:=\nu(Q)^{-1}\nu|_Q$ the normalization of $\nu$ restricted on $Q$.

$\widehat{f}(\xi):=\int e^{-2\pi i x\cdot\xi} f(x)\,dx$ is the Fourier transform; $f^\vee(\xi):=\int e^{2\pi i x\cdot\xi} f(x)\,dx$ is the inverse Fourier transform.

The direction of a tube $T$, denoted by $\theta(T)$, is the direction of its long axis. We say a tube $T$ is parallel to a cap in $S^{d-1}$ if its direction matches the center of the cap. Denote by $2T$ its dilation by factor $2$ from the center of $T$. 

For $x\neq y$, let $l_{x,y}$ denote the line connecting $x,y$. Denote by $T_{r_1\times r_2}(x;y)$ the $r_1\times r_2$-tube centered at $x$ of direction $l_{x,y}$.

\subsection*{Organization}
This paper is organized as follows. In Section \ref{sec-prelim}, we review different dimensions and some standard techniques in modern harmonic analysis. In Section \ref{sec-Good-Bad}, we first review the single-scale Good-Bad decomposition of Guth-Iosevich-Ou-Wang, then introduce our multi-scale Good-Bad decomposition and prove the $L^1$-estimate of our bad part. In Section \ref{sec-L2}, we first give a heuristic proof of Proposition \ref{thm-warmup}, then give a detailed proof of the $L^2$-estimate of our good part. In Section \ref{sec-proof-main}, we use our estimates on good and bad parts to prove Theorem \ref{thm-main}.

\subsection*{Acknowledgement}
The author would like to thank Tuomas Orponen for comments and suggestions on an earlier draft that helped improve the manuscript.

\section{Preliminaries}\label{sec-prelim}

Let $\mathcal{M}(E)$ denote the family of finite Borel measures supported on a compact subset of $E$. There are two equivalent ways to define Hausdorff dimension of sets $E\subset \R^d$ in terms of measures:
$$\begin{aligned}\dH E&=\sup_s\{s: \exists \,\mu\in\mathcal{M}(E) \text{ with }\mu(B(x,r))<r^s, \forall x\in\R^d, r>0\}\\&=\sup_s\{s: \exists \,\mu\in\mathcal{M}(E) \text{ with }I_s(\mu)<\infty\},\end{aligned}$$
where $I_s(\mu)$ is the $s$-dimension energy of $\mu$ defined by
$$I_s(\mu):=\iint |x-y|^{-s}d\mu(x)d\mu(y)= c_{s,d}\int|\hat{\mu}(\xi)|^2|\xi|^{-d+s}d\xi.$$
A finite Borel measure $\mu$ satisfying
$$\mu(B(x,r))\lesssim r^s,\,\forall x\in\R^d, r>0,$$
is called a Frostman measure of dimension $s$.

Let $N(E, r)$ denote the smallest number of $r$-balls needed to cover $E$. The upper Minkowski dimension (also called the upper box-counting dimension) is defined by
$$\overline{\dim}_M E:=\limsup_{r\rightarrow 0}\frac{\log N(E, r)}{\log 1/r}.$$
The packing dimension of $E$, denoted by $\dim_P E$, is defined by
$$\dim_P E:=\inf \{\sup_i \overline{\dim}_M E_i: E=\bigcup_{i=1}^\infty E_i, E_i\text{ is bounded}\}.$$
In particular, for every finite Borel measure $\mu$ on $E$ and every $\epsilon>0$, there exists $E_i\subset E$ such that $\overline{\dim}_M E_i<\dim_P E+\epsilon$ and $\mu(E_i)>0$.

It is easy to verify that $\dH E\leq \dim_P E\leq \overline{\dim}_M E$, and $\dH E=\dim_P E$ if $E$ is Ahlfors-David regular (not vice versa). For details and more discussion on their relations, we refer to Mattila's books \cite{Mat95} as a reference. 
 
We will need several standard techniques in modern harmonic analysis. We only state the version for our use. The proofs are not hard, so we omit details and refer to, for example, Guth's lecture notes (see his homepage) 
as a reference.

Throughout this paper, $\D_r$ denotes the collection of almost disjoint half-open-half-closed $r$-cubes that cover $\R^d$, $R^\epsilon Q$ denotes the dilation by factor $R^\epsilon$ from the center of $Q$, $m$ denotes the Lebesgue measure, and $m_Q:=m(Q)^{-1}m|_Q$ denotes the normalized Lebesgure measure on $Q$.

\begin{itemize}
    \item  (Localization) Suppose $\supp \hat{f}\subset B_R$ and $\nu$ is a finite Borel measure. Then for all $\epsilon>0$,
    $$\int |f|^2\,d\nu\lesssim_\epsilon R^\epsilon \sum_{Q\in\D_{R^{-1}}}\nu(R^\epsilon Q)\int |f|^2\,dm_{R^\epsilon Q}+\RD(R)\|f\|_{L^2(m)}^2.$$
    \item (Local $L^2$-orthogonality) Suppose $\{f_i\}$ is a family of functions with Fourier transform supported on $R$-balls of bounded overlapping. Then on each $R^{-1}$-ball $B_{R^{-1}}$,
    $$\int|\sum f_i|^2dm_{B_{R^{-1}}}\lesssim_{\epsilon}R^\epsilon \sum\int |f_i|^2dm_{R^\epsilon B_{R^{-1}}}+\RD(R)\sum\|f_i\|_{L^2(m)}^2.$$
    \item (Local constancy) Suppose $\hat{f}$ is supported in a $R_1\times\cdots\times R_d$ rectangle $T$ in $\R^d$, then for every dual rectangle $T^*$, that is a $R_1^{-1}\times \cdots \times R_d^{-1}$ rectangle whose $R_j^{-1}$-axis is parallel to the $R_j$-axis of $T$, we have, for all $\epsilon>0$,
    $$\|f\|^2_{L^\infty(T^*)}\lesssim_\epsilon R^\epsilon \|f\|_{L^2(m_{R^\epsilon T^*})}^2+\RD(R)\|f\|_{L^2(m)}^2 .$$
\end{itemize}

\section{A multi-scale Good-Bad decomposition}\label{sec-Good-Bad}
\subsection{The single-scale Good-Bad decomposition of Guth-Iosevich-Ou-Wang}\label{subsec-Good-Bad-GIOW}
For nearly 30 years, the $L^2$-method has been the most efficient approach to the Falconer distance conjecture, until Guth-Iosevich-Ou-Wang \cite{GIOW18} pointed out that it cannot beat Wolff's partial result in the plane \cite{Wol99}. To go further, they proposed the novel Good-Bad decomposition as follows. 

Let $R_0$ be a large dyadic number and $R_i:=2^iR_0$. For each $i\geq 1$, cover the unit circle by $R_i^{-1/2}$-caps $\theta$, and for each $\theta$, find a $R_i^{1/2} \times 10R_i$ rectangle parallel to $\theta$ such that the union of these rectangles covers the annulus $|\xi|\approx R_i$. Let $\{\psi_{i, \theta}\}$ be a partition of unity subordinate to this cover and write
\begin{equation}
	\label{def-psi-i-theta}
	1 = \psi_0 + \sum_{i \geq 1}\sum_{\theta : \,R_i^{-1/2}\text{-caps}}\psi_{i, \theta}.
\end{equation} 

Let $\delta > 0$ be a small fixed constant.

For each $(i, \theta)$, cover the unit ball $B_1$ with $R_i^{-1/2 + \delta} \times 1$ tubes $T$ parallel to $\theta$.  Denote $\T_{i,\theta}$ as the collection of these tubes, and let $\eta_T$ be a partition of unity subordinate to this cover such that
$$1=\sum_{T \in \T_{i, \theta}} \eta_T \text{ on }B_2.$$ 
Also denote
\begin{equation}\label{def-T-i}
    \T_i:=\bigcup_{\theta: \,R_i^{-1/2}\text{-caps}}\T_{i,\theta}.
\end{equation}

Now, suppose $\mu$ is a finite Borel measure on $E\subset B_1$ and $\nu$ is a finite Borel  measure on $F\subset B_1$ of disjoint supports. We say a tube $T \in \T_i$ is bad if
$$ \nu (T) \geq R_i^{-1/2 + 100 \delta},$$
and good otherwise.

Define $M_0 \mu := (\psi_0 \hat{\mu})^{\vee}$,
\begin{equation}\label{def-M-T} M_T \mu := \eta_T (\psi_{i, \theta} \hat{\mu})^{\vee}, \text{ if } T \in \T_{i, \theta},\, i\geq 1,\end{equation}
and
$$\mu_{good} := M_0 \mu + \sum_{i\geq 1}\sum_{T \in \T_i, \, good} M_T \mu.$$
If one can show that
\begin{equation}
    \label{GIOW-L1-bad}
    \|d_*^y(\mu_{good})-d_*^y(\mu))\|_{L^1(\R\times \nu)}=\|\sum_{i\geq 1}\sum_{T \in \T_i, \, bad} d_*^y(M_T \mu)\|_{L^1(\R\times\nu)}<1/100
\end{equation}
and
\begin{equation}
    \label{GIOW-L2-good}
 \|d_*^y(\mu_{good})\|_{L^2(\R\times\nu)}<\infty,  
\end{equation}
then
$$\begin{aligned}1=&\int\left(\int_{\Delta_y(E)} d^y_*(\mu)\right)d\nu(y)\\\leq &\|d_*^y(\mu_{good})-d_*^y(\mu))\|_{L^1}+( \sup_{y\in\supp\nu}|\Delta_y(E)|^{1/2})\|d_*^y(\mu_{good})\|_{L^2}\end{aligned}$$
implies that 
$$|\Delta_y(E)|>0,\ \text{for some }y\in F.$$

In \cite{GIOW18}, under the assumption $\dH F>1$ and $3\dH E+ \dH F> 5$, the bad part estimate \eqref{GIOW-L1-bad} follows from a radial projection theorem of Orponen \cite{Orp19}: suppose $\pi^x(y):=\frac{y-x}{|y-x|}$ and $t\in(1,2)$, then there exists $p>1$ such that
\begin{equation}\label{Orponen-Lp}\|\pi^x_*\nu\|_{L^p(S^1\times\mu)}\lesssim I_t(\nu)^{\frac{1}{2}}I_{2-t}(\mu)^{\frac{1}{2p}};\end{equation}
and the good part estimate \eqref{GIOW-L2-good} follows from a remarkable refined decoupling inequality. We point out that the constraint $\dH F>1$ in their result arises from the radial projection estimate \eqref{Orponen-Lp}.

\subsection{A multi-scale Good-Bad decomposition}\label{subsec-Good-Bad}

We call the Good-Bad decomposition of Guth-Iosevich-Ou-Wang ``single-scale", as only $R_i^{-1/2+\delta}\times 1$ tubes are considered for each $i\geq 1$.

In this paper, we propose a multi-scale Good-Bad decomposition. Though inspired by Guth-Iosevich-Ou-Wang, our Good-Bad decomposition is not a direct generalization of theirs. Heavy tubes are removed in different ways. In \cite{GIOW18}, Orponen's $L^p$ radial projection estimate \eqref{Orponen-Lp} is used to remove heavy tubes on $F$, while in this paper, it is used to remove heavy tubes on $E$, then an $L^2$-technique removes heavy tubes on $F$. See the proof of Proposition \ref{prop-good-bad} below for details. This is also why $\dH F\leq 1$ is allowed in our Theorem \ref{thm-main}.

Fix $i\geq 1$, let $R:=R_i$ be as in the previous subsection, and $\delta>0$ be a small constant that will be clarified later. Consider the sequence
\begin{equation}\label{def-r-j}
r_j:=R^{2^j\delta},\ j=0,\dots,\log_2 1/\delta. \end{equation}

Divide the unit sphere into $r_0^{-1}$-caps $\theta_0$, then divide each $\theta_0$ into $r_1^{-1}$-caps $\theta_1\subset\theta_0$,..., and finally into $r_{\log1/\delta-1}^{-1}=R^{-1/2}$-caps $\theta_{\log1/\delta-1}$. We may assume that $\{\theta_{\log1/\delta-1}\}$ coincides with  $\{\theta\}$ in the definition of $\T_i$ in \eqref{def-T-i}.

Recall $\nu$ is a finite Borel measure on $F$. We may also assume $\nu(F)=1$. For each cube $Q$ with $\nu(Q)\neq 0$, let
$$\nu_{Q}:=\nu(Q)^{-1}\nu|_{Q}$$ 
denote the normalization of $\nu$ restricted on $Q$. Below we only consider cubes with nonzero measure.

For each $r_j^{-1}$-cap $\theta_j$ and each dyadic $r_j/R$-cube $Q_j\in\D_{r_j/R}$, we say $\theta_j$ is bad for $Q_j$ if the $r_j/R\times r_{j+1}/R$ tube $T_{r_j/R\times r_{j+1}/R}$ centered at the center of $Q_j$ of direction $\theta_j$ satisfies
    \begin{equation}\label{def-bad-Q-r-j}\nu_{R^{2(j+2)\delta^2}Q_{j+1}}(R^{\delta}T_{r_j/R\times r_{j+1}/R})>R^{10\delta}\cdot (r_j/r_{j+1})^{{\min\{t, 1\}}},\end{equation}
where $Q_{j+1}\in\D_{r_{j+1}/R}$ is the dyadic $r_{j+1}/R$-cube containing $Q_j$, and $t$ is the dimension of our regular pin set $F$. 

For each $Q_j\in\D_{r_j/R}$, denote
\begin{equation}
	\label{def-Theta-j-Q-j}
	\Theta_j(Q_j):=\{\theta_j: r_j^{-1}\text{-caps, bad for } Q_j\}.
\end{equation}
More generally, for $Q_{j'}\in\D_{r_{j'}/R}, Q_{j}\in\D_{r_j/R}, Q_{j'}\subset Q_j, j'\leq j$, define
\begin{equation}
	\label{def-Theta-j-Q-j'}\Theta_j(Q_{j'}):=\Theta_j(Q_{j}).
\end{equation}

Recall $\T_i$ from \eqref{def-T-i}. For each dyadic $R_i^{-1+\delta}$-cube $Q\in \D_{r_0/R}$ and $j=0,\dots,\log\delta^{-1}-1$, denote
\begin{equation}\label{def-T-i-Q-j}\T_{i, Q,j}:=\{T\in\T_i: 2T\cap Q\neq\emptyset \text{ and } \exists\, \theta_j\in\Theta_j(Q) \text{ with } \theta(T)\in\theta_j \}.\end{equation}
Then let
\begin{equation}\label{def-T-i-Q}\T_{i, Q}:=\bigcup_{0\leq j\leq \log\delta^{-1}-1}\T_{i, Q,j}.\end{equation}

Now we are ready to decompose $\mu$ into good and bad parts. Recall the definition of $M_0\mu, M_T\mu$ from \eqref{def-M-T}.

First, for each $i\geq 1$ and each dyadic $R_i^{-1+\delta}$-cube $Q\in\D_{r_0/R}$, define
\begin{equation}\label{def-good-i-Q}\mu_{good, i, Q}:= \sum_{T\in \T_i\backslash \T_{i, Q}} M_T \mu.\end{equation}
Then, let $i\geq 1$ vary, and for each $y\in\supp\nu$, define
\begin{equation}\label{def-mu-good-y}\mu_{good, y}:=M_0\mu+\sum_{i\geq 1}\mu_{good, i, Q(y)},\end{equation}
where $Q(y)$ in each $\mu_{good, i, Q(y)}$ denotes the dyadic $R_i^{-1+\delta}$-cube containing $y$ .

Compared with the single-scale case, our $\mu_{good, y}$ depends on $y$. This is because the associated bad caps are different for different $y\in \supp\nu$. It brings additional difficulties on the $L^2$-estimate of the good part in the next section because we need to integrate over $y$. On the other hand, the $L^1$-estimate of the bad part below in this section is straightforward because it is a pointwise estimate with $y$ fixed. 

\begin{lem}\label{lem-3.5-GIOW}
For each $y\in\supp\nu$,
$$\begin{aligned}\|d_*^y(\mu_{good, y})-d_*^y(\mu)\|_{L^1(\R)}&=\|\sum_{i\geq 1}\sum_{T\in\T_{i, Q(y)}}d_*^y(M_T\mu)\|_{L^1(\R)}\\&\leq \sum_{i\geq 1}\sum_{T\in\T_{i, Q(y)}}\|d_*^y(M_T\mu)\|_{L^1(\R)}\\
    &\lesssim \sum_{i \geq 1} \mu ( Bad_{i}(y)) + \RD(R_0),\end{aligned}$$
    where,
    $$Bad_{i}(y) := \bigcup_{T \in \T_{i, Q(y)}} 4T. $$
\end{lem}
\begin{proof}
	[Proof of Lemma \ref{lem-3.5-GIOW}]
	It suffices to show that, for each $y\in\supp\nu$ and $T\in\T_{i,\theta}$,
	$$\|d_*^y(M_T\mu)\|_{L^1(\R)}\lesssim \begin{cases}
		\mu(4T)+\RD(R_i), & y\in 4T\\\RD(R_i), & y\notin 4T
	\end{cases},$$
	which is already proved in Section 3 in \cite{GIOW18}. We sketch the proof for completeness.
	
	By the definition of $M_T \mu$ in \eqref{def-M-T},
	$$M_T \mu(\cdot):= \eta_T (\psi_{i, \theta} \hat{\mu})^{\vee}(\cdot)=\eta_T(\cdot)\int \hat{\psi_{i, \theta}}(x-\cdot) d\mu(x),$$
	and then
	\begin{equation}\label{integral-d-y-M}d^y_*(M_T \mu)(t)= t\int \int_{S^1}\eta_T(y+t\sigma)\,\hat{\psi_{i, \theta}}(x-y-t\sigma) \, d\mu(x)\,d\sigma.\end{equation}
	Notice that $\eta_T$ is essentially supported in a $R_i^{-1/2+\delta}\times 1$ tube $T$ and $\hat{\psi_{i, \theta}}$ is essentially supported in a $R_i^{-1/2}\times R_i^{-1}$ tube perpendicular to $T$ centered at the origin. 
	
	The first observation from the essential supports of $\eta_T$ and $\hat{\psi_{i, \theta}}$ is, the integral \eqref{integral-d-y-M} equals $\RD(R_i)$ if $x\notin 4T$. Therefore, for all $y\in\R^2$,
	 $$\begin{aligned}\|d^y_*(M_T \mu)\|_{L^1(\R)}\lesssim & \int_{4T} \left(\int \eta_T(z)\,|\hat{\psi_{i, \theta}}(x-z)|\,dz\right) d\mu(x)+\RD(R_i)\\\lesssim &\mu(4T)+\RD(R_i).\end{aligned}$$
	 
	 On the other hand, when $y\in\supp\nu\backslash 4T$, the support of $\eta_T$ implies that the nontrivial integral domain of $\sigma$ in \eqref{integral-d-y-M} is $R_i^{-1/2}$ separated from the direction of $T$. Since $|x-y|\approx 1$, there is no way for $x-y-t\sigma$ to lie in the essential support of $\hat{\psi_{i, \theta}}$, that leads to $\RD(R_i)$ for the integral \eqref{integral-d-y-M}, and finally
	 $$\|d^y_*(M_T \mu)\|_{L^1(\R)}=\RD(R_i), \ y\in\supp\nu\backslash 4T.$$
	 
\end{proof}

Then, with
\begin{equation}\label{def-Bad-i}\begin{aligned}Bad_i:=&\{(x,y)\in \supp\mu\times \supp\nu:x\in Bad_i(y)\}\\=&\{(x,y)\in \supp\mu\times \supp\nu: \exists \,T\in  \T_{i, Q(y)},\, x\in 2T\},\end{aligned}\end{equation}
 Lemma \ref{lem-3.5-GIOW} implies that
   $$\begin{aligned}\int \|d_*^y(\mu_{good,y})-d_*^y(\mu))\|_{L^1(\R)}d\nu(y)\lesssim & \sum_{i \geq 1} \int \mu (Bad_i(y))d\nu(y)+ \RD(R_0)\\= &\sum_{i \geq 1} \mu\times\nu (Bad_i)+ \RD(R_0).\end{aligned}$$
Then we need the following geometric observation on pairs $(x,y)\in Bad_i$: by our definition of ``bad" in \eqref{def-bad-Q-r-j}-\eqref{def-T-i-Q} and \eqref{def-Bad-i}, there exist $0\leq j\leq \log1/\delta-1$ and $T\in\T_i$ such that $x,y\in 10T$ and $\theta(T)$ lies a $r_j^{-1}$-cap $\theta_j$ bad for $Q_j(y)\in\D_{r_j/R}$. Therefore, if we use
$$T_{r_j/R\times r_{j+1}/R}(y;x)$$
to denote the $r_j/R\times r_{j+1}/R$ tube centered at $y$ parallel to the line $l_{x,y}$ connecting $x,y$, then 
$$R^{2\delta}T_{r_j/R\times r_{j+1}/R}(y;x)\supset R^{\delta}T_{r_j/R\times r_{j+1}/R},$$
where $T_{r_j/R\times r_{j+1}/R}$ is the $r_j/R\times r_{j+1}/R$-tube centered at the center of $Q_j(y)$ of direction $\theta_j$ from \eqref{def-bad-Q-r-j}. Consequently,
$$\nu_{R^{2(j+2)\delta^2}Q_{j+1}}(R^{2\delta}T_{r_j/R\times r_{j+1}/R}(y;x))>R^{10\delta}\cdot (r_j/r_{j+1})^{\min\{t, 1\}},$$
and therefore the set $Bad_i$ is a subset of
$$\{(x,y)\in E\times F: \exists \,j,\,\nu_{R^{2(j+2)\delta^2}Q_{j+1}(y)}(R^{2\delta}T_{r_j/R\times r_{j+1}/R}(y;x))>R^{10\delta}\cdot (r_j/r_{j+1})^{\min\{t, 1\}}\}.$$
Moreover, $Bad_i$ is a subset of the construction in Proposition \ref{prop-good-bad} below with $R=R_i$ (it seems easier to compare their complement). Hence, if $\mu$ has finite $s$-energy with $s>1$ and $\nu$ is Frostman of dimension $t>2-s$ with $\overline{\dim}_M \supp\nu<t+\delta^2$, we have
\begin{equation}\label{measure-Bad-i}\mu\times\nu(Bad_i)\lesssim (\log R)^2 R^{-\delta^3}\end{equation}
and finally conclude the following.
\begin{prop}\label{prop-our-L1-bad}
    Suppose $\mu,\nu$ are probability measures on the unit ball with disjoint support, $I_s(\mu)<\infty$ for some $s>1$, and $\nu$ is
   Frostman of dimension $t>2-s$ with $\overline{\dim}_M \supp\nu<t+\delta^2$. Then, when $R_0$ is large enough and $\delta$ is small enough in terms of $s,t$,
    $$\int \|d_*^y(\mu_{good,y})-d_*^y(\mu))\|_{L^1(\R)}d\nu(y)< 1/100.$$
\end{prop}

Here is a more general result than what we need for \eqref{measure-Bad-i}.
\begin{prop}\label{prop-good-bad}
	Suppose $\mu,\nu$ are probability measures on $[0,1)^2$ with disjoint support, $I_s(\mu)<\infty$ for some $s>1$ and $\nu$ is
   Frostman of dimension $t>2-s$ with $\overline{\dim}_M \supp\nu<t+\delta^2$. Then, for all $R>1$ and $\delta>0$ small enough in terms of $s,t$, there exists a subset $Bad\subset [0,1)^2$ such that
	\begin{itemize}
		\item $\mu\times\nu(Bad)\lesssim (\log R)^2 R^{-\delta^3}$, and
		\item for all pairs of points $(x,y)\notin Bad$ and all dyadic numbers $r,r'$ with $R^{-1+\delta}\leq R^\delta r'\leq r\leq 1$, the $r'\times r$ tube $T_{r'\times r}(y;x)$ centered at $y$ parallel to the line $l_{x,y}$ connecting $x,y$ satisfies
		$$\nu_Q(T_{r'\times r}(y;x))< R^\delta\cdot (r'/r)^{\min\{t,1\}},$$
		where $Q=Q(y)\in\D_r$ is the dyadic $r$-cube containing $y$.
	\end{itemize}
	
\end{prop}

\begin{proof}[Proof of Proposition \ref{prop-good-bad}]
Fix $r',r$.

{\bf Step 1.} We first consider pairs $(x,y)\in\supp\mu\times\supp\nu$ with
$$\mu(T_{r/r'\times 1}(x;y))>R^{\delta^2}\cdot r'/r,$$
and denote the set of such pairs by $Bad_{r, r',\mu}$, which is an open subset of $[0,1)^2$. The estimate on $\mu\times\nu(Bad_{r, r',\mu})$ is the same as Lemma 3.6 in  \cite{GIOW18}, just with $\mu,\nu$ swapped. We give the proof for completeness. 

For each $y\in\supp\nu$, as $\dist(y, \supp\mu)\approx 1$, one can cover $\supp\mu$ by $\approx r/r'$ many $r'/r\times 1$ tubes containing $y$ of bounded overlapping. Denote this collections of tubes by $\T_{y, r'/r\times 1}$. Then, for each $x$ with $(x,y)\in Bad_{r, r',\mu}$, there exists $T\in \T_{y, r'/r\times 1}$ such that 
$$T_{r/r'\times 1}(x;y)\subset 2T.$$
Because of bounded overlapping, there are $\lesssim R^{-\delta^2}\cdot r/r'$ many $T\in \T_{y, r'/r\times 1}$ with
$$\mu(2T)> R^{\delta^2}\cdot r'/r.$$
Consequently, for each $y\in\supp\nu$,
$$\{x:(x,y)\in Bad_{r, r',\mu}\}$$
is contained in a union of $\lesssim R^{-\delta^2}\cdot r/r'$ many $r'/r\times 1$ tubes containing $y$. Therefore, if we denote $\pi^y(x):=\frac{y-x}{|y-x|}$ as the radial projection, then
$$\mu\times\nu(Bad_{r, r',\mu})=\int \mu(\{x:(x,y)\in Bad_{r, r',\mu}\})\,d\nu(y)= \int \left(\int \pi^y_*\mu\right) d\nu(y),$$
and for each $y\in\supp\nu$, the support of $\pi^y_*(\mu)$ is contained in $\lesssim R^{-\delta}\cdot r/r'$ many $r'/r$-arcs in $S^1$. Hence, since $s>1$, by Orponen's radial projection estiamte \eqref{Orponen-Lp} (with $\mu,\nu$ swapped), there exists $p>1$ such that
$$\mu\times\nu(Bad_{r, r',\mu})\lesssim R^{-\delta^2/p'} \cdot \|\pi^y_*\mu\|_{L^p(S^1\times\nu)}\lesssim R^{-\delta^2/p'}I_s(\mu)^{\frac{1}{2}}I_{2-s}(\nu)^{\frac{1}{2p}}\lesssim R^{-C \delta^2},$$
where the constant $C$ in the exponent only depends on $s,t$.

As a remark, the existence of some $p>1$ is sufficient here. We refer to \cite{Liu24} for a discussion on the range of $p$.

\medskip
{\bf Step 2.} Remove those $Q\in\D_r$ with $\nu(Q)<R^{-\delta^2}r^t$. More precisely, as 
$$\overline{\dim}_M\supp\nu<t+\delta^2,$$
the covering number satisfies
$$N(\supp\nu, r)\lesssim r^{-t-\delta^2},$$
which implies 
$$\#\{Q\in\D_r: Q\cap \supp\nu\neq\emptyset\}\lesssim r^{-t-\delta^2}.$$
Therefore,
$$\sum_{Q\in\D_r,\,\nu(Q)<R^{-\delta^2}r^t}\nu(Q)\lesssim r^{-t-\delta^2}\cdot R^{-\delta^2}r^t\lesssim R^{(1-\delta)\delta^2-\delta^2}\lesssim R^{-\delta^3}.$$

Together with Step 1, if we denote
$$Bad_{r, r',\mu, \nu}:=Bad_{r, r',\mu}\cup \left([0,1)\times \bigcup_{Q\in\D_r,\,\nu(Q)< R^{-\delta^2}r^t}Q\right),$$
then
$$\mu\times\nu(Bad_{r, r',\mu, \nu})\lesssim R^{-\delta^3}$$
when $\delta>0$ is small enough in terms of $s,t$.

\medskip
{\bf Step 3.} On $[0,1)^2\backslash Bad_{r, r',\mu,\nu}$, write 
$$\mu\times\nu=\sum_{Q\in\D_r}\nu(Q)\cdot\mu\times\nu_Q.$$

For each $Q\in\D_r$ and $y\in Q$, consider pairs $(x,y)\in[0,1)\times Q$ with
\begin{equation}\label{bad-r'-r-tube}\nu_Q(T_{r'\times r}(y;x))> R^{\delta}\cdot (r'/r)^{\min\{t,1\}},\end{equation}
and denote the set of these pairs by $Bad_{r, r',Q}\subset (Bad_{r, r', \mu, \nu})^c$. 

Choose $\approx r/r'$ many $r'/r$-separated directions $\{e_i\}\subset S^1$. For each selected direction $e_i$, cover $Q$ by $\approx r/r'$ many $r'\times r$ tubes parallel to $e_i$ of bounded overlapping. Denote the collection of all selected $r'\times r$ tubes by $\T_{Q, r'\times r}$.

Here comes the key observation in this step: for each $(x,y)\in Bad_{r, r',Q}$, there exist $T\in \T_{Q, r'\times r}$, an associated $r'/r\times 1$ tube of the same central line of $T_{r'/r\times 1}$, denoted by $T_{r'/r\times 1}\sim T$, such that 
$$T_{r'\times r}(y;x)\subset 2T \text{ and } x\in T_{r'/r\times 1}.$$

This observation implies that, if we call $T\in \T_{Q, r'\times r}$ bad when $$\nu(2T)> R^{\delta}\cdot (r'/r)^{\min\{t,1\}},$$ then
$$Bad_{r, r',Q}\subset \bigcup_{\substack{T\in \T_{Q, r'\times r}, \,bad\\T_{r'/r\times 1}\sim T}}T_{r'/r\times 1}\times T.$$
Since $(x,y)\notin Bad_{r, r', \mu}$, by our construction from Step 1, it follows that
$$\mu(T_{r'/r\times 1})< R^{\delta^2}\cdot r'/r.$$
Therefore,
$$\mu\times\nu_Q(Bad_{r, r',Q})\leq \sum_{\substack{T\in \T_{Q, r'\times r}, \,bad\\T_{r'/r\times 1}\sim T}}\mu(T_{r'/r\times 1})\nu_Q(T)\leq R^{\delta^2}\cdot r'/r \sum_{T\in \T_{Q, r'\times r}, \,bad}\nu_Q(T)$$
$$=R^{\delta^2}\cdot r'/r \sum_{\substack{e_i\in S^1 \\r'/r\text{-separated}}} \sum_{\substack{T\in \T_{Q, r'\times r}\\T\parallel e_i, \,bad}} \nu_Q(T).$$
As tubes in $\T_{Q, r'\times r}$ have bounded overlapping, by \eqref{bad-r'-r-tube}, for each $e_i$, the number of $T$, in the sum is $\lesssim R^{-\delta} (r/r')^{\min\{t,1\}}$. So, by Cauchy-Schwarz,
$$\begin{aligned}&\left(\mu\times\nu_Q(Bad_{r, r',Q})\right)^2\\\lesssim &  (R^{\delta^2} r'/r)^2\cdot r/r'\cdot R^{-\delta} (r/r')^{\min\{t,1\}}\sum_{\substack{e_i\in S^1 \\r'/r\text{-separated}}} \sum_{\substack{T\in \T_{Q, r'\times r}\\T\parallel e_i, \,bad}} \nu_Q(T)^2\\
\leq &R^{2\delta^2-\delta} (r'/r)^{1-\min\{t,1\}} \sum_{T\in \T_{Q, r'\times r}} \nu_Q(T)^2\\
= &R^{2\delta^2-\delta} (r'/r)^{1-\min\{t,1\}} \iint \sum_{T\in \T_{Q, r'\times r}}\chi_{T}(y)\chi_{T}(y')\, d\nu_Q(y)\,d\nu_Q(y').
\end{aligned}$$

Fix $y'$ and consider $|y-y'|\approx 2^{-j}$ with $2^{-j}\in[r', r]$. Then the number of $T\in \T_{Q, r'\times r}$ containing both $y,y'$ is $\lesssim 2^jr$. Also, since $\nu$ is Frostman of dimension $t$ and $\nu(Q)\geq R^{-\delta^2}r^t$ from Step 2, it follows that
$$\nu_Q(B(y, 2^{-j}))\lesssim R^{\delta^2} (2^jr)^{-t}.$$
Together we have
$$\begin{aligned}\left(\mu\times\nu_Q(Bad_{r, r',Q})\right)^2 \lesssim & R^{4\delta^2-\delta} (r'/r)^{1-\min\{t,1\}} \left((r'/r)^t+\sum_{j:2^{-j}\in[r', r]} 2^jr\cdot (2^jr)^{-t}\right)\\\lesssim &R^{4\delta^2-\delta},\end{aligned}$$
which is $\lesssim R^{-\delta^2}$ when $\delta<1/10$.

\medskip
{\bf Step 4.} Finally take
$$Bad:=\bigcup_{\substack{r',r:\, dyadic\\R^{-1+\delta}\leq R^\delta r'\leq r\leq 1}} \bigcup_{Q\in\D_r}Bad_{r,r',Q}\cup Bad_{r,r',\mu,\nu}.$$
From previous steps,
$$\mu\times\nu(Bad)\lesssim \sum_{r,r'}\left(R^{-\delta^3}+\sum_{Q\in\D_r}\nu(Q)R^{-\delta^3}\right)\lesssim (\log R)^2 R^{-\delta^3}$$
when $\delta>0$ is small enough in terms of $s,t$, and all required conditions are satisfied.
\end{proof}

\section{Multi-scale Mizohata-Takeuchi-type estimates}\label{sec-L2}

\subsection{A  heuristic proof of Proposition \ref{thm-warmup}}\label{sec-proof-warmup}
The idea behind Proposition \ref{thm-warmup} is an $L^2$ ball inflation. A heuristic proof is given in this subsection that ignores many $\delta$-power-loss. We use the symbol $\lessapprox$ for this ignorance. A rigorous argument on this idea will be given in the next subsection.

Let $\psi\in C_0^\infty$ be non-negative with $\hat{\psi}\neq 0$ on the unit ball. Denote $\psi_R(\cdot):=R^d\psi(R \cdot)$. Then, as the integral is over $B_R$,
$$\int_{B_R}|\widehat{f\,d\sigma}|^2w\lesssim \int_{B_R}|\widehat{(f\,d\sigma)*\psi_R}|^2w:=\int_{B_R}|F|^2w.$$
 For each $0\leq j\leq \log1/\delta-1$, divide the unit sphere into  $r_j^{-1}$-caps $\theta_j$ and let $\psi_{\theta_j}$ be a partition of unity subordinate to this cover. Denote
$$F_{\theta_j}:= \widehat{(f\psi_{\theta_j}d\sigma)*\psi_R}.$$

First, by the localization, we may assume $w$ is a constant on each $r_0$-cube, namely,
\begin{equation}\label{heuristic-localization}\|F\|_{L^2(w)}^2\lessapprox \sum_{Q_0\in\D_{r_0}}w(Q_0)\int |F|^2dm_{Q_0}.\end{equation}

Second, as the supports of $\widehat{F_{\theta_0}}$ lie in almost disjoint $r_0^{-1}$-balls, the local $L^2$-orthogonality between $F_{\theta_0}$s on each $Q_0\in\D_{r_0}$ implies
\begin{equation}\label{heuristic-orthogonality}\int |F|^2dm_{Q_0}\lessapprox\sum_{\theta_0: r_0^{-1}\text{-caps}}\int |F_{\theta_0}|^2dm_{Q_0}.\end{equation}

Third, as $\widehat{F_{\theta_0}}$ is supported on a $r_0^{-1}\times\cdots\times r_0^{-1}\times r_1^{-1}$ rectangle perpendicular to $\theta_0$, the function $F_{\theta_0}$ is locally a constant on each $r_0\times \cdots\times r_0\times r_1$ tube parallel to $\theta_0$. Therefore, if we cover $Q_1\in\D_{r_1}$ by $r_0\times \cdots\times r_0\times r_1$ tubes $T_0$ parallel to $\theta_0$ with bounded overlapping, then
$$\begin{aligned}\sum_{Q_0\subset Q_1}w(Q_0)\int |F_{\theta_0}|^2dm_{Q_0}& =\sum_{T_0\parallel\theta_0}\sum_{Q_0\subset T_0}w(Q_0)\int |F_{\theta_0}|^2dm_{Q_0}\\
&\lessapprox \sum_{T_0\parallel\theta_0}\sum_{Q_0\subset T_0}w(Q_0)\|F_{\theta_0}\|_{L^\infty(T_0)}^2\\
&=\sum_{T_0\parallel\theta_0}w(T_0)\|F_{\theta_0}\|_{L^\infty(T_0)}^2\\
&\lessapprox w(Q_1)\sum_{T_0\parallel\theta_0}\frac{w_{Q_1}(T_0)}{m_{Q_1}(T_0)}\int_{T_0} |F_{\theta_0}|^2dm_{Q_1},
\end{aligned}$$
which is, because of $m_{Q_1}(T_0)\approx (r_1/r_0)^{d-1}$,
$$\lessapprox (r_0/r_1)^{d-1} \left(\sup_{\substack{\forall\,Q_{1}\in\D_{r_1}\\\forall\,T_0\in\T_{r_0, r_1}}}w_{Q_1}(T_0)\right)w(Q_1)\int |F_{\theta_0}|^2dm_{Q_1}.$$

Together with \eqref{heuristic-localization} and \eqref{heuristic-orthogonality}, we obtain an $L^2$ ball inflation that reduces the spatial scale from $r_0$ to $r_1$:
$$\|F\|_{L^2(w)}^2\lessapprox\sum_{Q_0\in\D_{r_0}}w(Q_0)\int |F|^2dm_{Q_0}$$
$$\lessapprox (r_1/r_0)^{d-1} \left(\sup_{\substack{\forall\,Q_{1}\in\D_{r_1}\\\forall\,T_0\in\T_{r_0, r_1}}}w_{Q_1}(T_0)\right)\sum_{\theta_0}\sum_{Q_1\in\D_{r_1}}w(Q_1)\int |F_{\theta_0}|^2dm_{Q_1}.$$

By iterating this process, we end up with
$$\|F\|_{L^2(w)}^2\lessapprox R^{d-1} \left(\prod_{j=0}^{M-1} \sup_{\substack{\forall\,Q_{j+1}\in\D_{r_{j+1}}\\\forall\,T_j\in\T_{r_j, r_{j+1}}}} w_{Q_{j+1}}(T_j)\right)w(B_R)\sum_{\theta: R^{-1/2}\text{-caps}} \int|F_\theta|^2dm_{B_R},$$
and Proposition \ref{thm-warmup} follows by the standard ``Plancherel \& Cauchy-Schwarz'' argument:
$$\begin{aligned}\sum_{\theta: R^{-1/2}\text{-caps}} \int|F_\theta|^2dm_{B_R}\lesssim & \,R^{-d} \sum_{\theta: R^{-1/2}\text{-caps}} \int|\widehat{(f\psi_{\theta}d\sigma)*\psi_R}|^2dm\\
=& \,R^{-d} \sum_{\theta: R^{-1/2}\text{-caps}} \int|(f\psi_{\theta}d\sigma)*\psi_R|^2dm\\
 \lesssim & \,R^{-(d-1)}\sum_{\theta: R^{-1/2}\text{-caps}} \int|f\psi_\theta(\sigma)|^2\left(\int |\psi_R(x-\sigma)|dx\right)d\sigma\\ \lesssim &\,R^{-(d-1)}\int|f|^2d\sigma.\end{aligned}$$

\subsection{$L^2$ estimates of the good part}\label{subsec-L2-good}

\begin{prop}\label{prop-our-L2-good}
    Suppose $\mu,\nu$ are probability measures on the unit ball with disjoint support, $I_s(\mu)<\infty$ for some $s>1$, and $\nu$ is
   Frostman of dimension $t>2-s$ with $\overline{\dim}_M \supp\nu<t+\delta^2$. Then, with $\mu_{good,y}$ defined in \eqref{def-mu-good-y}, when $\delta>0$ is small enough in terms of $s,t$,
$$\int \| d^y_*(\mu_{good,y}) \|_{L^2(\R)}^2d\nu(y)<\infty.$$    
\end{prop}

By the observation in \cite{Liu18}, for each $y\in\supp\nu$,
$$ \| d^y_*(\mu_{good,y}) \|_{L^2(\R)}^2 = \int_0^\infty | \mu_{good,y} * \sigma_t (y)|^2 t^2 dt,$$
where $\sigma_t$ denotes the normalized arc-length measure on $tS^1$. This is
$$\lesssim \int_0^\infty | \mu_{good,y} * \sigma_t (y)|^2 t dt $$
because $\mu_{good,y}$ is essentially supported near $\supp\mu$, away from $y\in\supp\nu$ (see Lemma 5.2 in \cite{GIOW18}). Then, by the $L^2$-identity proved in \cite{Liu18}, for every $y \in \R^2$, 
$$\int_0^\infty | \mu_{good,y} * \sigma_t (y)|^2 t dt=\int_0^\infty | \mu_{good,y} * \hat{\sigma_r} (y)|^2 r dr.$$
By the definition of $\mu_{good, i, Q(y)}$, $\mu_{good, y}$ in \eqref{def-good-i-Q}, \eqref{def-mu-good-y}, it suffices to show
$$\int_0^\infty | \mu_{good,i, Q(y)} * \hat{\sigma_r} (y)|^2 r dr\approx \int_{r\approx R_i} | \mu_{good,i, Q(y)} * \hat{\sigma_r} (y)|^2 rdr+\RD(R_i)$$
is summable in $i\geq 1$, where the right hand side follows because the Fourier support of $\mu_{good,i, Q(y)}$ lies in the annulus $|\xi|\approx R_i$.

Notice that the dependence on $y$ in $\mu_{good, y}$ has not yet made any difference because so far everything is pointwise. 

Now we fix $R\approx R_i$ and consider the integral with respect to $d\nu(y)$:
\begin{equation}\label{L2-good-formulation}\int | \mu_{good,i, Q(y)} * \hat{\sigma_R} (y)|^2 d\nu(y)\lesssim \int | ((\widehat{\mu_{good,i, Q(y)}} d\sigma_R)*\psi)^\vee (y)|^2 d\nu(y),\end{equation}
where $\psi\in C_0^\infty$ is non-negative with $\hat{\psi}\neq 0$ on the unit ball, and the right hand side follows because $\supp\nu\subset B_1$.

To deal with the dependence on $Q(y)$, a careful analysis on bad caps associated with different $y$ is required. 

Let $r_j$ be as in \eqref{def-r-j} and recall the definition of $\Theta_j$ in \eqref{def-Theta-j-Q-j}, \eqref{def-Theta-j-Q-j'}. Notice that \begin{equation}\label{compare-Theta-j}\Theta_j(Q_{j_1})=\Theta_j(Q_{j_2})=\Theta_j(Q_j)\end{equation}
for all $Q_{j_1}\in \D_{r_{j_1}/R}$, $Q_{j_2}\in \D_{r_{j_2}/R}$ contained in the same $Q_{j}\in \D_{r_{j}/R}$, $\forall j_1, j_2\leq j$.

For $0\leq k<j\leq\log1/\delta-1$ and $Q_k\in\D_{r_k/R}$, let
\begin{equation}\label{def-Theta-=j}\Theta_{=j}(Q_k):=\{\theta_j\in\Theta_{j}(Q_k): \text{disjoint with all caps in }\Theta_k(Q_k),\dots,\Theta_{j-1}(Q_k)\}.\end{equation}

Now we can rewrite the integrand in \eqref{L2-good-formulation} to the following.

\begin{lem}
    \label{lem-rewrite-good-part}
    Recall $\psi_{i, \theta}$ defined in \eqref{def-psi-i-theta}, and let
$$F:= ((\hat{\mu}\,d\sigma_R)*\psi)^\vee, \ F_\theta:=((\psi_{i, \theta}\hat{\mu}\,d\sigma_R)*\psi)^\vee,\ F_{\theta_j}:=\sum_{\theta\subset\theta_j} F_\theta.$$
Then for every $Q\in\D_{r_0/R}$ and $y\in Q$,
$$((\widehat{\mu_{good,i, Q}} d\sigma_R)*\psi)^\vee (y)$$$$=F(y)-\sum_{\theta_0\in\Theta_0(Q)} F_{\theta_0}(y)-\sum_{j=1}^{\log1/\delta-1} \sum_{\theta_j\in\Theta_{=j}(Q)} F_{\theta_j}(y)+\RD(R).$$
\end{lem}
\begin{proof}[Proof of Lemma \ref{lem-rewrite-good-part}]
    First, by the definition of $\T_{i, Q, j}, \T_{i, Q}, \mu_{good, i, Q}$ in \eqref{def-T-i-Q-j}-\eqref{def-good-i-Q},
    $$\mu_{good, i, Q}=\sum_{T \in \T_i} M_T\mu-\sum_{T \in \T_{i, Q, 0}} M_T \mu-\sum_{j=1}^{\log1/\delta-1}\sum_{\substack{T \in \T_{i, Q, j}\\\theta(T)\in\Theta_{=j}(Q)}} M_T \mu,$$
where $\theta(T)\in\Theta_{=j}(Q)$ means $\theta(T)$ lies in some $\theta_j\in \Theta_j(Q)$. Then, by the definition of $\T_{i, Q, j}$ in \eqref{def-T-i-Q-j}, Lemma \ref{lem-rewrite-good-part} follows from a basic property of wave packets:
\begin{equation}\label{fast-decay-on-Q}\| ((\widehat{M_T \mu} \,d\sigma_R)*\psi)^\vee\|_{L^\infty(Q)}=\RD(R), \text{ given }2T\cap Q=\emptyset.\end{equation}
More precisely, by the definition of $M_T \mu$ in \eqref{def-M-T}, one can write out
$$((\widehat{M_T \mu} \,d\sigma_R)*\psi)^\vee(y)= \psi^\vee(y)\int\left(\int_{S^1}\int e^{2\pi i (y\cdot R\sigma-x\cdot R\sigma+x\cdot\xi)}\,\eta_T(x)dxd\sigma \right)  \psi_{i, \theta}(\xi) \hat{\mu}(\xi)\,d\xi,$$
then integration by parts on $x$ shows that it is $\RD(R)$ unless $\sigma$ lies in the $R^{-1/2}$-neighborhood of $\theta$, and finally, because the angle between $\theta$ and $l_{x,y}, \forall x\in T, \forall y\in Q$ is $>R^{-1/2+\delta}$, integration by parts on $\sigma$ under local coordinates concludes \eqref{fast-decay-on-Q}. We refer to \cite{Ste93}, Chapter VIII, for details of the second integration by parts, and \cite{Demeter20} as a reference for wave packets.
\end{proof}

For each $0\leq k\leq \log 1/\delta $, each $r_k/R$-cube $Q_k\in\D_{r_k/R}$, and each $r_{k-1}^{-1}$-cap $\theta_{k-1}$, let
\begin{equation}\label{def-F-k}F_{Q_k,\theta_{k-1}}:=
	F_{\theta_{k-1}}-\sum_{\theta_k\in\Theta_k(Q_k), \theta_k\subset\theta_{k-1}} F_{\theta_k}-\sum_{j=k+1}^{\log1/\delta-1} \sum_{\theta_j\in\Theta_{=j}(Q_k), \theta_j\subset\theta_{k-1}} F_{\theta_j},\end{equation}
with
$$\theta_{-1}:=S^1,\ \sum_{j=\log1/\delta}^{\log1/\delta-1}=\sum_{j=\log1/\delta+1}^{\log1/\delta-1}:=0,\ \Theta_{\log 1/\delta}:=\emptyset$$ as convention. In particular, when $k=0$, it becomes the main term in Lemma \ref{lem-rewrite-good-part}; when $k=\log 1/\delta$, it becomes $F_{\theta_{k-1}}$ for $R^{-1/2}$-caps $\theta_{k-1}$. We shall need the following relation between $F_{Q_k,\theta_{k-1}}$ and $F_{Q_{k+1},\theta_k}$.
\begin{lem}
	\label{lem-relation}
	Suppose $Q_k\subset Q_{k+1}$, $Q_k\in\D_{r_k/R}$, $Q_{k+1}\in\D_{r_{k+1}/R}$, then
	 $$F_{Q_k,\theta_{k-1}}(y)=\sum_{\theta_k\notin\Theta_k(Q_k),\,\theta_k\subset\theta_{k-1}} F_{Q_{k+1}, \theta_k}(y),\ \forall\,y\in\R^2.$$
\end{lem}

\begin{proof}[Proof of Lemma \ref{lem-relation}]
Notice that
\begin{equation}\label{rewrite-F-Q-k-theta-k-1-1}
\begin{aligned}F_{Q_k,\theta_{k-1}}=&\sum_{\theta_k\subset\theta_{k-1}}F_{\theta_k}-\sum_{\theta_k\in\Theta_k(Q_k), \,\theta_k\subset\theta_{k-1}} F_{\theta_k}-\sum_{\theta_k\subset\theta_{k-1}}\sum_{j=k+1}^{\log1/\delta-1} \sum_{\theta_j\in\Theta_{=j}(Q_k),\, \theta_j\subset\theta_k} F_{\theta_j}\\=&\sum_{\theta_k\notin\Theta_k(Q_k),\,\theta_k\subset\theta_{k-1}}F_{\theta_k}-\sum_{\theta_k\subset\theta_{k-1}}\sum_{j=k+1}^{\log1/\delta-1} \sum_{\theta_j\in\Theta_{=j}(Q_k), \theta_j\subset\theta_k} F_{\theta_j}.
\end{aligned}\end{equation}
By the definition of $\Theta_{=j}$ in \eqref{def-Theta-=j}, the conditions $\theta_j\in\Theta_{=j}(Q_k), \theta_j\subset\theta_k, j\geq k+1$ imply $\theta_k\notin\Theta_k(Q_k)$. Therefore \eqref{rewrite-F-Q-k-theta-k-1-1} equals
\begin{equation}\label{rewrite-F-Q-k-theta-k-1-2}\sum_{\theta_k\notin\Theta_k(Q_k),\,\theta_k\subset\theta_{k-1} }\left(F_{\theta_k}-\sum_{j=k+1}^{\log1/\delta-1} \sum_{\theta_j\in\Theta_{=j}(Q_k), \theta_j\subset\theta_k} F_{\theta_j}\right).\end{equation}
As the sum is taken over $\theta_k\notin\Theta_k(Q_k)$, no $\theta_{k+1}\subset\theta_k$ intersect caps in $\Theta_k(Q_k)$. So \eqref{rewrite-F-Q-k-theta-k-1-2} can be written as
$$\sum_{\theta_k\notin\Theta_k(Q_k),\,\theta_k\subset\theta_{k-1}}(F_{\theta_k}-\sum_{\theta_{k+1}\in\Theta_{k+1}(Q_k), \,\theta_{k+1}\subset\theta_k} F_{\theta_k}-\sum_{j=k+2}^{\log1/\delta-1} \sum_{\theta_j\in\Theta_{=j}(Q_k), \theta_j\subset\theta_k} F_{\theta_j}).$$
 Since $\Theta_j(Q_k)=\Theta_j(Q_{k+1})$ when $Q_k\subset Q_{k+1}$ and $j\geq k+1$ (recall \eqref{compare-Theta-j}), it coincides with
 $$\sum_{\theta_k\notin\Theta_k(Q_k),\,\theta_k\subset\theta_{k-1}}\left(F_{\theta_k}-\sum_{\theta_{k+1}\in\Theta_{k+1}(Q_{k+1}), \,\theta_{k+1}\subset\theta_k} F_{\theta_k}-\sum_{j=k+2}^{\log1/\delta-1} \sum_{\theta_j\in\Theta_{=j}(Q_{k+1}), \theta_j\subset\theta_k} F_{\theta_j}\right)$$
  $$=\sum_{\theta_k\notin\Theta_k(Q_k),\,\theta_k\subset\theta_{k-1}} F_{Q_{k+1}, \theta_k},$$
as desired.
\end{proof}

Now, by Lemma \ref{lem-rewrite-good-part} and notation \eqref{def-F-k}, the integral \eqref{L2-good-formulation} equals, up to a negligible error,
\begin{equation}\label{rewrite-L2-good-formulation}\sum_{Q_0\in\D_{r_0/R}}\nu(Q_0)\int |F_{Q_0,\theta_{-1}}(y)|^2\,d\nu_{Q_0}(y),\end{equation}
which is, by the standard localization argument,
\begin{equation}\label{localize-L2-good} \lesssim R^{C\delta}\sum_{Q_0\in\D_{r_0/R}}\nu(R^{2\delta^2}Q_0)\int |F_{Q_0, \theta_{-1}}(y)|^2\,dm_{R^{2\delta^2}Q_0}(y).\end{equation}
In this paper, the constant $C>0$ in the exponent of $R$ may vary from line to line but is independent in $\delta$ and $R$.

The following $L^2$ ball inflation lemma will eventually reduce the spatial scale from $r_0/R$ to $1$.
\begin{lem}[$L^2$ ball inflation]\label{lem-reduction}
For $0\leq k\leq \log 1/\delta-1 $, each $Q_{k+1}\in\D_{r_{k+1}/R}$, and each $r_{k-1}^{-1}$-cap $\theta_{k-1}$,
$$\sum_{Q_{k}\subset Q_{k+1}}\nu(R^{2(k+1)\delta^2}Q_{k})\int | F_{Q_k, \theta_{k-1}}(y)|^2\,dm_{R^{2(k+1)\delta^2}Q_{k}}(y) $$	
$$\lesssim R^{C\delta} (r_k/r_{k+1})^{\min\{t,1\}-1}\sum_{\theta_{k}\subset\theta_{k-1}}\nu(R^{2(k+2)\delta^2}Q_{k+1})\int |F_{Q_{k+1}, \theta_k}(y)|^2\,dm_{R^{2(k+2)\delta^2}Q_{k+1}}(y).$$
\end{lem}

\begin{proof}
	[Proof of Lemma \ref{lem-reduction}]
	The idea is the same as the heuristic proof of Proposition \ref{thm-warmup} in Section \ref{sec-proof-warmup}, and the dependence on $Q_k\in \D_{r_k/R}$ will be reduced to $Q_{k+1}\in \D_{r_{k+1}/R}$ in a natural way. We shall give a detailed proof. I hope that the heuristic argument in the previous subsection helps for a better understanding of this idea.
	
By Lemma \ref{lem-relation},  for each $Q_k\in\D_{r_k/R}$,
$$\int | F_{Q_k, \theta_{k-1}}(y)|^2\,dm_{R^{2(k+1)\delta^2}Q_{k}}(y)=\int |\sum_{\theta_k\notin\Theta_k(Q_k),\,\theta_k\subset\theta_{k-1}} F_{Q_{k+1}, \theta_k}(y)|^2\,dm_{R^{2(k+1)\delta^2}Q_{k}}(y).$$
As the Fourier supports of $\{F_{Q_{k+1},\theta_k}\}$ lie in $r_k^{-1}$-caps of $RS^1$ with bounded overlapping, they are contained in $Rr_k^{-1}$-balls centered in $RS^1$ of bounded overlapping. So by the local $L^2$-orthogonality, for each $r_k/R$-cube $Q'\subset R^{2(k+1)\delta^2} Q_{k}$,
$$\int|\sum_{\theta_{k}\notin\Theta_{k}(Q_k), \,\theta_k\subset\theta_{k-1} } F_{Q_{k+1},\theta_k}|^2dm_{Q'}\lesssim R^{C\delta} \sum_{\theta_k\notin\Theta_k(Q_k),\, \theta_k\subset\theta_{k-1} }\int|F_{Q_{k+1},\theta_k}|^2\,dm_{R^{\delta^2} Q'},$$
up to an error term $\RD(R)\|F\|_{L^2}^2$. Since there are $\lesssim R^{C\delta^2}$ many such $Q'$, it follows that, up to a negligible error,
\begin{equation}\label{local-orthogonal}\begin{aligned}\int |F_{Q_k, \theta_{k-1}}|^2\,dm_{R^{2(k+1)\delta^2}Q_k}= &\int |\sum_{\theta_k\notin\Theta_k(Q_k),\,\theta_k\subset\theta_{k-1}} F_{Q_{k+1}, \theta_k}|^2\,dm_{R^{2(k+1)\delta^2}Q_{k}}\\\lesssim &\,  R^{C\delta}\sum_{\theta_k\notin \Theta_k(Q_k), \,\theta_k\subset\theta_{k-1}}\int | F_{Q_{k+1}, \theta_k}|^2\,dm_{R^{(2k+3)\delta^2}Q_k}.\end{aligned}\end{equation}

By taking the sum of \eqref{local-orthogonal} over $Q_k\subset Q_{k+1}$ with weight $\nu(R^{2(k+1)\delta^2}Q_k)$, one can see that the left hand side in Lemma \ref{lem-reduction} satisfies
\begin{equation}\label{the-sum}\begin{aligned}
& \sum_{Q_{k}\subset Q_{k+1}}\nu(R^{2(k+1)\delta^2}Q_{k})\int | F_{Q_k, \theta_{k-1}}|^2\,dm_{R^{2(k+1)\delta^2}Q_{k}}\\
\lesssim &\  R^{C\delta}\sum_{Q_k\subset Q_{k+1}}\sum_{\theta_k\notin \Theta_k(Q_k), \,\theta_k\subset\theta_{k-1}}\nu(R^{2(k+1)\delta^2}Q_k)\int | F_{Q_{k+1}, \theta_k}|^2\,dm_{R^{(2k+3)\delta^2}Q_k}.\end{aligned}\end{equation}

Now we fix $\theta_k\subset\theta_{k-1}$ and count $Q_k$. As $\theta_k\notin\Theta_k(Q_k)$, by the definition of $\Theta_k(Q_k)$ in \eqref{def-Theta-j-Q-j} and \eqref{def-bad-Q-r-j}, $Q_k$ is counted only if the $r_k/R\times r_{k+1}/R$ tube $T_{r_k/R\times r_{k+1}/R}$ centered at the center of $Q_k$ of direction $\theta_k$ is ``good", meaning
    \begin{equation}\label{good-k}\nu_{R^{2(k+2)\delta^2}Q_{k+1}}(R^\delta T_{r_k/R\times r_{k+1}/R})<R^{10\delta}\cdot (r_k/r_{k+1})^{\min\{t,1\}}.\end{equation}
So, if we cover $Q_{k+1}$ by $r_k/R\times r_{k+1}/R$ tubes $T$ parallel to $\theta_k$ of bounded overlapping, then \eqref{the-sum} is
\begin{equation}\label{rearrange-good-T}\lesssim R^{C\delta}\sum_{\theta_{k}\subset\theta_{k-1}}\sum_{\substack{T\parallel \theta_{k}\text{ with }\eqref{good-k}\\r_{k}/R\times r_{k+1}/R\text{-tubes} }}\sum_{Q_{k}\cap T\neq\emptyset}\nu(R^{2(k+1)\delta^2}Q_k)\int|F_{Q_{k+1},\theta_k}|^2\,dm_{R^{(2k+3)\delta^2}Q_{k}}.\end{equation}

For each $T$ in \eqref{rearrange-good-T}, cover $R^{(2k+3)\delta^2}T$ by $r_{k}/R\times r_{k+1}/R$ tubes $T'$ of the same direction. As the Fourier support of $F_{Q_{k+1},\theta_{k}}$ lies in a $R/r_{k}\times R/r_{k+1}$ rectangle containing $\theta_{k}$, it has locally constant property on each $T'$, and therefore
$$\begin{aligned}&\sum_{Q_{k}\cap T\neq\emptyset}\nu(R^{2(k+1)\delta^2}Q_{k})\int_{T'}|F_{Q_{k+1},\theta_{k}}|^2\,dm_{R^{(2k+3)\delta^2}Q_{k}}\\
\lesssim & R^{C\delta}\cdot \nu(R^{2(k+2)\delta^2} T)\cdot \|F_{Q_{k+1},\theta_k}\|_{L^\infty(T')}^2\\
= & R^{C\delta}\cdot \nu_{R^{2(k+2)\delta^2}Q_{k+1}}(R^{2(k+2)\delta^2} T) \cdot \nu(R^{2(k+2)\delta^2}Q_{k+1})\cdot \|F_{Q_{k+1},\theta_k}\|_{L^\infty(T')}^2\\
\lesssim & R^{C\delta}\cdot \nu_{R^{2(k+2)\delta^2}Q_{k+1}}(R^{\delta} T) \cdot \nu(R^{2(k+2)\delta^2}Q_{k+1})\cdot m(R^{\delta^2}T')^{-1}\int_{R^{\delta^2}T'}|F_{Q_{k+1},\theta_k}|^2\,dm,
\end{aligned}$$
which is, because $T$ satisfies \eqref{good-k} and $m(R^{\delta^2}T')\gtrsim R^{-C\delta}\cdot (r_k/r_{k+1})\cdot m(R^{2(k+2)\delta^2}Q_{k+1})$,
\begin{equation}\label{local-constant-r-0}
\lesssim R^{C\delta}\cdot(r_k/r_{k+1})^{\min\{t,1\}-1}\cdot \nu(R^{2(k+2)\delta^2}Q_{k+1})\int_{R^{\delta^2}T'}|F_{Q_{k+1},\theta_k}|^2\,dm_{R^{2(k+2)\delta^2}Q_{k+1}}.\end{equation}
Finally we take the sum of \eqref{local-constant-r-0} over $T'\subset R^{(2k+3)\delta^2}T$ and then the sum over $T, \theta_k$ in \eqref{rearrange-good-T} to obtain the following upper bound of the left hand side in Lemma \ref{lem-reduction}:
\begin{multline*}
R^{C\delta}(r_k/r_{k+1})^{\min\{t,1\}-1}\sum_{\theta_k\subset\theta_{k-1}}\nu(R^{2(k+2)\delta^2}Q_{k+1})\sum_{\substack{T\parallel \theta_k\text{ with }\eqref{good-k}\\r_k/R\times r_{k+1}/R\text{-tubes} }}\\\int_{R^{2(k+2)\delta^2}T}|F_{Q_{k+1},\theta_{k}}|^2\,dm_{R^{2(k+2)\delta^2}Q_{k+1}}.
\end{multline*}

Now we have finished all the work in the spatial scale $r_k/R$, can add those ``bad'' $r_k/R\times r_{k+1}/R$ tubes without \eqref{good-k} back to move to the next spatial scale $r_{k+1}/R$. More precisely, since the union of $r_k/R\times r_{k+1}/R$ tubes $T$ covers $Q_k$ of bounded overlapping, the above is
$$	\lesssim R^{C\delta}(r_k/r_{k+1})^{\min\{t,1\}-1}\sum_{\theta_k\subset\theta_{k-1}}\nu(R^{2(k+2)\delta^2}Q_{k+1})\int|F_{Q_{k+1},\theta_{k}}|^2\,dm_{R^{2(k+2)\delta^2}Q_{k+1}},$$
thus completes the proof of Lemma \ref{lem-reduction}.
\end{proof}

By \eqref{L2-good-formulation}, \eqref{rewrite-L2-good-formulation}, \eqref{localize-L2-good}, and iteration of Lemma \ref{lem-reduction}, we end up with
$$\begin{aligned} \int | \mu_{good,i, Q(y)} * \hat{\sigma_R} (y)|^2 d\nu(y)\lesssim &\,R^{C\delta\log1/\delta}\prod_{k=0}^{\log\delta-1}(r_k/r_{k+1})^{\min\{t,1\}-1}\sum_{\theta: R^{-1/2}\text{-caps}}\int_{B_{R^\delta}}|F_\theta|^2dm\\
\lesssim &\,R^{C\delta\log1/\delta} R^{1-\min\{t,1\}} \sum_{\theta: R^{-1/2}\text{-caps}}\int|F_\theta|^2dm.\end{aligned}$$

The rest is the standard ``Plancherel \& Cauchy-Schwarz'' argument:
$$\begin{aligned}\sum_{\theta: R^{-1/2}\text{-caps}}\int|F_\theta|^2dm &=\sum_{\theta: R^{-1/2}\text{-caps}} \int|((\psi_{i, \theta}\hat{\mu}\,d\sigma_R)*\psi)^\vee|^2\,dm\\
&= \sum_{\theta: R^{-1/2}\text{-caps}}\int|(\psi_{i, \theta}\hat{\mu}\,d\sigma_R)*\psi|^2dm\\
& \lesssim  R^{-1} \int_{S^1} \sum_{\theta: R^{-1/2}\text{-caps}}|(\psi_{i, \theta}\hat{\mu})(R\sigma)|^2\left(\int\psi(\xi-R\sigma)\,d\xi\right)d\sigma\\ &\lesssim  R^{-1}\int_{S^1}|\hat{\mu}(R\sigma)|^2\,d\sigma.\end{aligned}$$

Hence, 
$$\int | \mu_{good,i, Q(y)} * \hat{\sigma_R} (y)|^2 d\nu(y)\lesssim R^{C\delta\log1/\delta} R^{-\min\{t,1\}} \int_{S^1}|\hat{\mu}(R\sigma)|^2\,d\sigma,$$
and, after integrating over $R\approx R_i$, it follows that
$$\begin{aligned}\int \int_{r\approx R_i} | \mu_{good,y} * \hat{\sigma_r} (y)|^2 r dr\,d\nu(y)&\lesssim_\delta R_i^{C\delta\log1/\delta}R_i^{-\min\{t,1\}} \int_{|\xi|\approx R_i}|\hat{\mu}(\xi)|^2\,d\xi\\ &\lesssim_{\delta} R_i^{C\delta\log1/\delta}R_i^{2-s-\min\{t,1\}} I_{s}(\mu),\end{aligned}$$
summable in $i$ if $s>1$, $s+t>2$, and $\delta>0$ is small enough in terms of $s, t$.
This completes the proof of Proposition \ref{prop-our-L2-good}.

\section{Proof of the main Theorem}\label{sec-proof-main}
Now we can prove Theorem \ref{thm-main}. By the discussion in Section \ref{sec-prelim}, there exist a probability measure $\mu$ on $E$ of finite $s$-energy $I_s(\mu)$ for some $s>1$, and a probability measure $\nu$ on $F$ that is
   Frostman of dimension $t>2-s$ with $\overline{\dim}_M \supp\nu<t+\delta^2$. Then, by Proposition \ref{prop-our-L1-bad}, Proposition \ref{prop-our-L2-good}, and
$$\begin{aligned}1=&\int\left(\int_{\Delta_y(E)} d^y_*(\mu)\right)d\nu(y)\\\leq &\|d_*^y(\mu_{good, y})-d_*^y(\mu))\|_{L^1}+( \sup_{y\in\supp\nu}|\Delta_y(E)|^{\frac{1}{2}})\|d_*^y(\mu_{good, y})\|_{L^2},\end{aligned}$$
one can conclude that
$$|\Delta_y(E)|>0,\ \text{for some }y\in F,$$
that completes the proof of Theorem \ref{thm-main}.

\bibliographystyle{plainurl}
\bibliography{mybibtex.bib}

@misc{FP26+,
	archiveprefix = {arXiv},
	author = {Jonathan M. Fraser and Thang Pham},
	eprint = {2604.19486},
	primaryclass = {math.CA},
	title = {On Fourier decay and the distance set problem},
	url = {https://arxiv.org/abs/2604.19486},
	year = {2026}}

@article{AH76,
	author = {Agmon, S. and H\"ormander, L.},
	doi = {10.1007/BF02786703},
	fjournal = {Journal d'Analyse Math\'ematique},
	issn = {0021-7670,1565-8538},
	journal = {J. Analyse Math.},
	mrclass = {35B40},
	mrnumber = {466902},
	mrreviewer = {J.\ Muszy\'nski},
	pages = {1--38},
	title = {Asymptotic properties of solutions of differential equations with simple characteristics},
	url = {https://doi.org/10.1007/BF02786703},
	volume = {30},
	year = {1976}}

@article{FP15,
	author = {Fraser, Jonathan M. and Pollicott, Mark},
	doi = {10.1017/S0305004115000523},
	fjournal = {Mathematical Proceedings of the Cambridge Philosophical Society},
	issn = {0305-0041,1469-8064},
	journal = {Math. Proc. Cambridge Philos. Soc.},
	mrclass = {28A80 (28A78 37B50 37F35)},
	mrnumber = {3413889},
	mrreviewer = {Manuel\ Mor\'an},
	number = {3},
	pages = {547--566},
	title = {Micromeasure distributions and applications for conformally generated fractals},
	url = {https://doi.org/10.1017/S0305004115000523},
	volume = {159},
	year = {2015}}

@article{FFS15,
	author = {Ferguson, Andrew and Fraser, Jonathan M. and Sahlsten, Tuomas},
	doi = {10.1016/j.aim.2014.09.019},
	fjournal = {Advances in Mathematics},
	issn = {0001-8708,1090-2082},
	journal = {Adv. Math.},
	mrclass = {37C45 (28A80 28D05 37A30 37C40)},
	mrnumber = {3276605},
	mrreviewer = {Esa\ J\"arvenp\"a\"a},
	pages = {564--602},
	title = {Scaling scenery of {$(\times m,\times n)$} invariant measures},
	url = {https://doi.org/10.1016/j.aim.2014.09.019},
	volume = {268},
	year = {2015}}

@article{Barany17,
	author = {B\'ar\'any, Bal\'azs},
	doi = {10.4064/fm90-4-2016},
	fjournal = {Fundamenta Mathematicae},
	issn = {0016-2736,1730-6329},
	journal = {Fund. Math.},
	mrclass = {28A80 (28A78 37C45)},
	mrnumber = {3606399},
	mrreviewer = {Zhiyong\ Zhu},
	number = {1},
	pages = {83--100},
	title = {On some non-linear projections of self-similar sets in {$\Bbb{R}^3$}},
	url = {https://doi.org/10.4064/fm90-4-2016},
	volume = {237},
	year = {2017}}

@misc{Mulherkar25+,
	archiveprefix = {arXiv},
	author = {Siddharth Mulherkar},
	eprint = {2506.05624},
	primaryclass = {math.CA},
	title = {Random Constructions for Sharp Estimates of Mizohata-Takeuchi Type},
	url = {https://arxiv.org/abs/2506.05624},
	year = {2025}}

@misc{CLPY25+,
	archiveprefix = {arXiv},
	author = {Anthony Carbery and Zane Kun Li and Yixuan Pang and Po-Lam Yung},
	eprint = {2510.04345},
	primaryclass = {math.CA},
	title = {A weighted formulation of refined decoupling and inequalities of Mizohata-Takeuchi-type for the moment curve},
	url = {https://arxiv.org/abs/2510.04345},
	year = {2025}}

@article{CHV23,
	author = {Carbery, Anthony and H\"anninen, Timo S. and Valdimarsson, Stef\'an Ingi},
	doi = {10.2140/apde.2023.16.511},
	fjournal = {Analysis \& PDE},
	issn = {2157-5045,1948-206X},
	journal = {Anal. PDE},
	mrclass = {26D15 (42Bxx 46E30 47B38)},
	mrnumber = {4593773},
	number = {2},
	pages = {511--543},
	title = {Disentanglement, multilinear duality and factorisation for nonpositive operators},
	url = {https://doi.org/10.2140/apde.2023.16.511},
	volume = {16},
	year = {2023}}

@unpublished{Guth22,
	author = {Guth, Larry},
	note = {Joint talk for AIM Research Community `Fourier restriction conjecture and related problems' and HAPPY network.},
	title = {An enemy scenario in restriction theory},
	url = {https:// www.youtube.com/watch?v=x-DET83UjFg},
	year = {2022}}

@article{CIW24,
	author = {Carbery, Anthony and Iliopoulou, Marina and Wang, Hong},
	doi = {10.4171/rmi/1463},
	fjournal = {Revista Matem\'atica Iberoamericana},
	issn = {0213-2230,2235-0616},
	journal = {Rev. Mat. Iberoam.},
	mrclass = {42B37 (42B10)},
	mrnumber = {4759602},
	mrreviewer = {Ferit\ G\"urb\"uz},
	number = {4},
	pages = {1387--1418},
	title = {Some sharp inequalities of {M}izohata-{T}akeuchi-type},
	url = {https://doi.org/10.4171/rmi/1463},
	volume = {40},
	year = {2024}}

@article{CZ25+,
	archiveprefix = {arXiv},
	author = {Hannah Cairo and Ruixiang Zhang},
	eprint = {2512.08064},
	primaryclass = {math.CA},
	title = {Power loss for the {M}izohata-{T}akeuchi conjecture on {$C^k$} convex hypersurfaces},
	url = {https://arxiv.org/abs/2512.08064},
	year = {2025}}

@article{Cairo25+,
	archiveprefix = {arXiv},
	author = {Hannah Cairo},
	eprint = {2502.06137},
	primaryclass = {math.CA},
	title = {A Counterexample to the Mizohata-Takeuchi Conjecture},
	url = {https://arxiv.org/abs/2502.06137},
	year = {2025}}

@article{WZ22,
	author = {Wang, Zijian and Zheng, Jiqiang},
	doi = {10.4064/cm8632-10-2021},
	fjournal = {Colloquium Mathematicum},
	issn = {0010-1354,1730-6302},
	journal = {Colloq. Math.},
	mrclass = {28A75 (42B20)},
	mrnumber = {4491618},
	number = {2},
	pages = {171--191},
	title = {An improvement of the pinned distance set problem in even dimensions},
	url = {https://doi.org/10.4064/cm8632-10-2021},
	volume = {170},
	year = {2022}}

@article{DGOWWZ21,
	author = {Du, Xiumin and Guth, Larry and Ou, Yumeng and Wang, Hong and Wilson, Bobby and Zhang, Ruixiang},
	doi = {10.1353/ajm.2021.0005},
	fjournal = {American Journal of Mathematics},
	issn = {0002-9327,1080-6377},
	journal = {Amer. J. Math.},
	mrclass = {28A80 (28A75 42B10 42B20)},
	mrnumber = {4201782},
	mrreviewer = {Zolt\'an\ L.\ Buczolich},
	number = {1},
	pages = {175--211},
	title = {Weighted restriction estimates and application to {F}alconer distance set problem},
	url = {https://doi.org/10.1353/ajm.2021.0005},
	volume = {143},
	year = {2021}}

@article{Fraser23,
	author = {Fraser, Jonathan M.},
	doi = {10.1112/jlms.12697},
	fjournal = {Journal of the London Mathematical Society. Second Series},
	issn = {0024-6107,1469-7750},
	journal = {J. Lond. Math. Soc. (2)},
	mrclass = {28A80 (11B30 28A78)},
	mrnumber = {4549145},
	mrreviewer = {J\"org\ Neunh\"auserer},
	number = {2},
	pages = {777--797},
	title = {A nonlinear projection theorem for {A}ssouad dimension and applications},
	url = {https://doi.org/10.1112/jlms.12697},
	volume = {107},
	year = {2023}}

@article{Fraser18,
	author = {Fraser, Jonathan M.},
	doi = {10.1007/s11856-018-1715-z},
	fjournal = {Israel Journal of Mathematics},
	issn = {0021-2172,1565-8511},
	journal = {Israel J. Math.},
	mrclass = {28A80 (28A75)},
	mrnumber = {3819712},
	mrreviewer = {Lars\ Olsen},
	number = {2},
	pages = {851--875},
	title = {Distance sets, orthogonal projections and passing to weak tangents},
	url = {https://doi.org/10.1007/s11856-018-1715-z},
	volume = {226},
	year = {2018}}

@article{SW25,
	author = {Shmerkin, Pablo and Wang, Hong},
	doi = {10.1007/s00039-024-00696-5},
	fjournal = {Geometric and Functional Analysis},
	issn = {1016-443X,1420-8970},
	journal = {Geom. Funct. Anal.},
	mrclass = {28A78 (28A80)},
	mrnumber = {4865135},
	mrreviewer = {Lars\ Olsen},
	number = {1},
	pages = {283--358},
	title = {On the distance sets spanned by sets of dimension {$d/2$} in {$\Bbb R^d$}},
	url = {https://doi.org/10.1007/s00039-024-00696-5},
	volume = {35},
	year = {2025}}

@article{Stull22+,
	archiveprefix = {arXiv},
	author = {D. M. Stull},
	eprint = {2207.12501},
	primaryclass = {cs.CC},
	title = {Pinned Distance Sets Using Effective Dimension},
	url = {https://arxiv.org/abs/2207.12501},
	year = {2022}}

@article{FS23+,
	archiveprefix = {arXiv},
	author = {Jacob B. Fiedler and D. M. Stull},
	eprint = {2309.11701},
	primaryclass = {math.CA},
	title = {Dimension of Pinned Distance Sets for Semi-Regular Sets},
	url = {https://arxiv.org/abs/2309.11701},
	year = {2023}}

@article{FS25+,
	archiveprefix = {arXiv},
	author = {Jacob B. Fiedler and D. M. Stull},
	eprint = {2408.00889},
	primaryclass = {math.CA},
	title = {Pinned distances of planar sets with low dimension},
	url = {https://arxiv.org/abs/2408.00889},
	year = {2024}}

@article{Orponen12,
	author = {Orponen, Tuomas},
	doi = {10.1088/0951-7715/25/6/1919},
	fjournal = {Nonlinearity},
	issn = {0951-7715,1361-6544},
	journal = {Nonlinearity},
	mrclass = {28A80 (28A78 37C45)},
	mrnumber = {2929609},
	mrreviewer = {Xiang-Yang\ Wang},
	number = {6},
	pages = {1919--1929},
	title = {On the distance sets of self-similar sets},
	url = {https://doi.org/10.1088/0951-7715/25/6/1919},
	volume = {25},
	year = {2012}}

@article{DORZ23-2,
	archiveprefix = {arXiv},
	author = {Xiumin Du and Yumeng Ou and Kevin Ren and Ruixiang Zhang},
	eprint = {2309.04103},
	primaryclass = {math.CA},
	title = {New improvement to Falconer distance set problem in higher dimensions},
	url = {https://arxiv.org/abs/2309.04103},
	year = {2023}}

@article{DORZ23-1,
	archiveprefix = {arXiv},
	author = {Xiumin Du and Yumeng Ou and Kevin Ren and Ruixiang Zhang},
	eprint = {2309.04501},
	primaryclass = {math.CA},
	title = {Weighted refined decoupling estimates and application to Falconer distance set problem},
	url = {https://arxiv.org/abs/2309.04501},
	year = {2023}}

@book{Demeter20,
	author = {Demeter, Ciprian},
	doi = {10.1017/9781108584401},
	isbn = {978-1-108-49970-5},
	mrclass = {42-02 (11B30 11L07)},
	mrnumber = {3971577},
	mrreviewer = {Andreas\ Nilsson},
	pages = {xvi+331},
	publisher = {Cambridge University Press, Cambridge},
	series = {Cambridge Studies in Advanced Mathematics},
	title = {Fourier restriction, decoupling, and applications},
	url = {https://doi.org/10.1017/9781108584401},
	volume = {184},
	year = {2020}}

@article{Liu24,
	author = {Liu, Bochen},
	doi = {10.4171/rmi/1472},
	fjournal = {Revista Matem\'atica Iberoamericana},
	issn = {0213-2230,2235-0616},
	journal = {Rev. Mat. Iberoam.},
	mrclass = {28A75 (42B20)},
	mrnumber = {4739824},
	number = {3},
	pages = {827--858},
	title = {Mixed-norm of orthogonal projections and analytic interpolation on dimensions of measures},
	url = {https://doi.org/10.4171/rmi/1472},
	volume = {40},
	year = {2024}}

@article{DIOWZ20,
	author = {Du, Xiumin and Iosevich, Alex and Ou, Yumeng and Wang, Hong and Zhang, Ruixiang},
	doi = {10.1007/s00208-021-02170-1},
	fjournal = {Mathematische Annalen},
	issn = {0025-5831},
	journal = {Math. Ann.},
	mrclass = {28A80},
	mrnumber = {4297185},
	number = {3-4},
	pages = {1215--1231},
	title = {An improved result for {F}alconer's distance set problem in even dimensions},
	url = {https://doi.org/10.1007/s00208-021-02170-1},
	volume = {380},
	year = {2021}}

@article{Shm20,
	author = {Shmerkin, Pablo},
	journal = {J. Eur. Math. Soc. (2022), DOI:10.4171/JEMS/1283},
	title = {A nonlinear version of Bourgain's projection theorem}}

@article{KS18,
	author = {Keleti, Tam\'{a}s and Shmerkin, Pablo},
	doi = {10.1007/s00039-019-00500-9},
	fjournal = {Geometric and Functional Analysis},
	issn = {1016-443X},
	journal = {Geom. Funct. Anal.},
	mrclass = {28 (26A15 49Q15)},
	mrnumber = {4034924},
	number = {6},
	pages = {1886--1948},
	title = {New bounds on the dimensions of planar distance sets},
	url = {https://doi.org/10.1007/s00039-019-00500-9},
	volume = {29},
	year = {2019}}

@article{KT01,
	author = {Katz, Nets Hawk and Tao, Terence},
	fjournal = {New York Journal of Mathematics},
	issn = {1076-9803},
	journal = {New York J. Math.},
	mrclass = {28A80 (28A75 28A78)},
	mrnumber = {1856956},
	mrreviewer = {Miguel Angel Mart\'{\i}n},
	pages = {149--187},
	title = {Some connections between {F}alconer's distance set conjecture and sets of {F}urstenburg type},
	url = {http://nyjm.albany.edu:8000/j/2001/7_149.html},
	volume = {7},
	year = {2001}}

@article{Shm18,
	author = {Shmerkin, Pablo},
	doi = {10.4171/jfg/97},
	fjournal = {Journal of Fractal Geometry. Mathematics of Fractals and Related Topics},
	issn = {2308-1309},
	journal = {J. Fractal Geom.},
	mrclass = {28A75 (28A80 49Q15)},
	mrnumber = {4226184},
	number = {1},
	pages = {27--51},
	title = {Improved bounds for the dimensions of planar distance sets},
	url = {https://doi.org/10.4171/jfg/97},
	volume = {8},
	year = {2021}}

@article{Liu18-dimension,
	author = {Liu, Bochen},
	doi = {10.1090/proc/14740},
	fjournal = {Proceedings of the American Mathematical Society},
	issn = {0002-9939},
	journal = {Proc. Amer. Math. Soc.},
	mrclass = {28A75 (42B20)},
	mrnumber = {4042855},
	number = {1},
	pages = {333--341},
	title = {Hausdorff dimension of pinned distance sets and the {$L^2$}-method},
	url = {https://doi.org/10.1090/proc/14740},
	volume = {148},
	year = {2020}}

@article{Orp19,
	author = {Orponen, Tuomas},
	doi = {10.2140/apde.2019.12.1273},
	fjournal = {Analysis \& PDE},
	issn = {2157-5045},
	journal = {Anal. PDE},
	mrclass = {28A80 (28A78 42C10)},
	mrnumber = {3892404},
	number = {5},
	pages = {1273--1294},
	title = {On the dimension and smoothness of radial projections},
	url = {https://doi-org.easyaccess2.lib.cuhk.edu.hk/10.2140/apde.2019.12.1273},
	volume = {12},
	year = {2019}}

@article{GIOW18,
	author = {Guth, Larry and Iosevich, Alex and Ou, Yumeng and Wang, Hong},
	doi = {10.1007/s00222-019-00917-x},
	fjournal = {Inventiones Mathematicae},
	issn = {0020-9910},
	journal = {Invent. Math.},
	mrclass = {42B20 (28A80)},
	mrnumber = {4055179},
	number = {3},
	pages = {779--830},
	title = {On {F}alconer's distance set problem in the plane},
	url = {https://doi.org/10.1007/s00222-019-00917-x},
	volume = {219},
	year = {2020}}

@article{DZ18,
	author = {Du, Xiumin and Zhang, Ruixiang},
	doi = {10.4007/annals.2019.189.3.4},
	fjournal = {Annals of Mathematics. Second Series},
	issn = {0003-486X},
	journal = {Ann. of Math. (2)},
	mrclass = {42B20 (42B37)},
	mrnumber = {3961084},
	number = {3},
	pages = {837--861},
	title = {Sharp {$L^2$} estimates of the {S}chr\"{o}dinger maximal function in higher dimensions},
	url = {https://doi-org.easyaccess1.lib.cuhk.edu.hk/10.4007/annals.2019.189.3.4},
	volume = {189},
	year = {2019}}

@article{Liu18,
	author = {Liu, Bochen},
	doi = {10.1007/s00039-019-00482-8},
	fjournal = {Geometric and Functional Analysis},
	issn = {1016-443X},
	journal = {Geom. Funct. Anal.},
	mrclass = {Prelim},
	mrnumber = {3925111},
	number = {1},
	pages = {283--294},
	title = {An {$L^2$}-identity and pinned distance problem},
	url = {https://doi-org.easyaccess1.lib.cuhk.edu.hk/10.1007/s00039-019-00482-8},
	volume = {29},
	year = {2019}}

@article{shm17,
	author = {Shmerkin, Pablo},
	doi = {10.1007/s11856-019-1847-9},
	fjournal = {Israel Journal of Mathematics},
	issn = {0021-2172},
	journal = {Israel J. Math.},
	mrclass = {28A80 (28D20)},
	mrnumber = {3940442},
	number = {2},
	pages = {949--972},
	title = {On the {H}ausdorff dimension of pinned distance sets},
	url = {https://doi-org.easyaccess1.lib.cuhk.edu.hk/10.1007/s11856-019-1847-9},
	volume = {230},
	year = {2019}}

@book{Ste93,
	author = {Stein, Elias M.},
	isbn = {0-691-03216-5},
	mrclass = {42-02 (35Sxx 43-02 47G30)},
	mrnumber = {1232192},
	mrreviewer = {Michael Cowling},
	note = {With the assistance of Timothy S. Murphy, Monographs in Harmonic Analysis, III},
	pages = {xiv+695},
	publisher = {Princeton University Press, Princeton, NJ},
	series = {Princeton Mathematical Series},
	title = {Harmonic analysis: real-variable methods, orthogonality, and oscillatory integrals},
	volume = {43},
	year = {1993}}

@book{Mat95,
	author = {Mattila, Pertti},
	doi = {10.1017/CBO9780511623813},
	isbn = {0-521-46576-1; 0-521-65595-1},
	mrclass = {28A75 (49Q20)},
	mrnumber = {1333890},
	mrreviewer = {Harold Parks},
	pages = {xii+343},
	publisher = {Cambridge University Press, Cambridge},
	series = {Cambridge Studies in Advanced Mathematics},
	title = {Geometry of Sets and Measures in Euclidean Spaces: Fractals and Rectifiability},
	url = {http://dx.doi.org/10.1017/CBO9780511623813},
	volume = {44},
	year = {1995}}

@book{Fal86,
	author = {Falconer, K. J.},
	isbn = {0-521-25694-1; 0-521-33705-4},
	mrclass = {28-02 (28A99 60J65)},
	mrnumber = {867284},
	mrreviewer = {K. E. Hirst},
	pages = {xiv+162},
	publisher = {Cambridge University Press, Cambridge},
	series = {Cambridge Tracts in Mathematics},
	title = {The geometry of fractal sets},
	volume = {85},
	year = {1986}}

@article{Bou03,
	author = {Bourgain, Jean},
	coden = {GFANFB},
	doi = {10.1007/s000390300008},
	fjournal = {Geometric and Functional Analysis},
	issn = {1016-443X},
	journal = {Geom. Funct. Anal.},
	mrclass = {11K55 (28A80)},
	mrnumber = {1982147},
	mrreviewer = {Ben Joseph Green},
	number = {2},
	pages = {334--365},
	title = {On the {E}rd{\H o}s-{V}olkmann and {K}atz-{T}ao ring conjectures},
	url = {http://dx.doi.org/10.1007/s000390300008},
	volume = {13},
	year = {2003}}

@article{PS00,
	author = {Peres, Yuval and Schlag, Wilhelm},
	coden = {DUMJAO},
	doi = {10.1215/S0012-7094-00-10222-0},
	fjournal = {Duke Mathematical Journal},
	issn = {0012-7094},
	journal = {Duke Math. J.},
	mrclass = {42B25 (28A78)},
	mrnumber = {1749437},
	mrreviewer = {Esa A. J{\"a}rvenp{\"a}{\"a}},
	number = {2},
	pages = {193--251},
	title = {Smoothness of projections, {B}ernoulli convolutions, and the dimension of exceptions},
	url = {http://dx.doi.org/10.1215/S0012-7094-00-10222-0},
	volume = {102},
	year = {2000}}

@article{Orp17,
	author = {Orponen, Tuomas},
	doi = {10.1016/j.aim.2016.11.035},
	fjournal = {Advances in Mathematics},
	issn = {0001-8708},
	journal = {Adv. Math.},
	mrclass = {28A80 (28A78)},
	mrnumber = {3590535},
	pages = {1029--1045},
	title = {On the distance sets of {A}hlfors-{D}avid regular sets},
	url = {http://dx.doi.org/10.1016/j.aim.2016.11.035},
	volume = {307},
	year = {2017}}

@article{IL16,
	author = {Iosevich, Alex and Liu, Bochen},
	doi = {doi:10.5186/aasfm.2016.4135},
	fjournal = {Annales Academi\ae\ Scientiarum Fennic\ae . Mathematica},
	issn = {1239-629X},
	journal = {Ann. Acad. Sci. Fenn. Math.},
	mrclass = {28A75 (52C10)},
	mrnumber = {3525385},
	number = {2},
	pages = {579--585},
	title = {Falconer distance problem, additive energy and Cartesian products},
	volume = {41},
	year = {2016}}

@article{Mat87,
	author = {Mattila, Pertti},
	doi = {10.1112/S0025579300013462},
	fjournal = {Mathematika. A Journal of Pure and Applied Mathematics},
	issn = {0025-5793},
	journal = {Mathematika},
	mrclass = {42B10 (28A12)},
	mrnumber = {933500},
	mrreviewer = {Michael Cowling},
	number = {2},
	pages = {207--228},
	title = {Spherical averages of {F}ourier transforms of measures with finite energy; dimension of intersections and distance sets},
	url = {http://dx.doi.org/10.1112/S0025579300013462},
	volume = {34},
	year = {1987}}

@article{Fal85,
	author = {Falconer, K. J.},
	coden = {MTKAAB},
	doi = {10.1112/S0025579300010998},
	fjournal = {Mathematika. A Journal of Pure and Applied Mathematics},
	issn = {0025-5793},
	journal = {Mathematika},
	mrclass = {28A75 (28A05)},
	mrnumber = {834490},
	mrreviewer = {S. J. Taylor},
	number = {2},
	pages = {206--212},
	title = {On the {H}ausdorff dimensions of distance sets},
	url = {http://dx.doi.org/10.1112/S0025579300010998},
	volume = {32},
	year = {1985}}

@article{Bou94,
	author = {Bourgain, Jean},
	coden = {ISJMAP},
	doi = {10.1007/BF02772994},
	fjournal = {Israel Journal of Mathematics},
	issn = {0021-2172},
	journal = {Israel J. Math.},
	mrclass = {28A78},
	mrnumber = {1286826},
	mrreviewer = {K. E. Hirst},
	number = {1-3},
	pages = {193--201},
	title = {Hausdorff dimension and distance sets},
	url = {http://dx.doi.org/10.1007/BF02772994},
	volume = {87},
	year = {1994}}

@article{Wol99,
	author = {Wolff, Thomas},
	doi = {10.1155/S1073792899000288},
	fjournal = {International Mathematics Research Notices},
	issn = {1073-7928},
	journal = {Internat. Math. Res. Notices},
	mrclass = {42B10 (42B25)},
	mrnumber = {1692851},
	mrreviewer = {Steen Pedersen},
	number = {10},
	pages = {547--567},
	title = {Decay of circular means of {F}ourier transforms of measures},
	url = {http://dx.doi.org/10.1155/S1073792899000288},
	year = {1999}}

@article{Erd05,
	author = {Erdo\u{g}an, M. Burak},
	doi = {10.1155/IMRN.2005.1411},
	fjournal = {International Mathematics Research Notices},
	issn = {1073-7928},
	journal = {Int. Math. Res. Not.},
	mrclass = {42B20 (28A75)},
	mrnumber = {2152236},
	mrreviewer = {Herv{\'e} Pajot},
	number = {23},
	pages = {1411--1425},
	title = {A bilinear {F}ourier extension theorem and applications to the distance set problem},
	url = {http://dx.doi.org/10.1155/IMRN.2005.1411},
	year = {2005}}
\end{document}